\documentclass[10pt]{amsart}
\usepackage{amssymb}
\usepackage{amsmath}
\usepackage{amsfonts, color}
\DeclareMathOperator*{\slim}{s-lim}
\usepackage{graphicx}
\usepackage{geometry}
\usepackage{mathrsfs,amssymb, slashed, cite}

\usepackage{hyperref}

\usepackage{cleveref}
\renewcommand{\H}{\mathcal{H}}
\newcommand{\dd}{\mathop{}\!\mathrm{d}}

\newcommand{\La}{\mathcal{L}_a}
\newcommand{\la}{\lambda}
\newcommand{\R}{\mathbb{R}}
\newcommand{\C}{\mathbb{C}}

\let\Re=\undefined\DeclareMathOperator*{\Re}{Re}
\let\Im=\undefined\DeclareMathOperator*{\Im}{Im}

\newcommand{\supp}{\text{supp}}

\renewcommand{\L}{\mathcal{L}}

\newcommand{\Z}{\mathbb{Z}}

\newtheorem{definition}{Definition}

\newtheorem{lemma}{Lemma}[section]

\newtheorem{proposition}{Proposition}[section]
\newtheorem{remark}{Remark}[section]

\newtheorem{theorem}{Theorem}[section]
\numberwithin{equation}{section}

\begin{document}
	\author{Mingming Deng}
	\address{Mingming Deng
		\newline \indent School of Mathematics and Statistics, Zhengzhou University,
		Zhengzhou, Henan, 450001, China}
	\email{dengmingming@zzu.edu.cn}
	\author{Yilin Song}
	\address{Yilin Song
		\newline \indent Institute of Applied Physics and Computational Mathematics,
		Beijing 100088,\ P. R. China}
	\email{songyilin21@gscaep.ac.cn}
\author{Ruixiao Zhang}
\address{Ruixiao Zhang
\newline\indent Department of Mathematics, \hfill\newline Northeastern University,
	\hfill\newline Shenyang, 110819,  People's Republic of China.}
\email{zhangrx@mail.neu.edu.cn}

	\title[Bilinear estimates for Schr\"odinger equation]{Bilinear estimates for Schr\"odinger equation with repulsive inverse-square potential and its application}
	
	\begin{abstract} \noindent
		In this paper, we establish the bilinear Strichartz estimates for the following linear Schr\"odinger equation with repulsive inverse-square potential $i\partial_tu+\La u=0$ where $\La=-\Delta+a|x|^{-2}$ with $a\geq0$ for non-radial data. Our proof relies on the physical space method (interaction Morawetz estimate). Such type of method was first introduced in Planchon-Vega [Ann. Sci. \'Ec. Norm. Sup\'er. (4) {\bf 42} (2009), 261--290] in flat case,  and Planchon-Tzvetkov-Visciglia [Rev. Mat. Iberoam,  \textbf{39}(2023), no. 4, 1405-1436] for harmonic oscillator.  Our bilinear estimate  extends results obtained in Camps [Ann. Inst. H. Poincar\'e Anal. Non Lin\'eaire {\bf 39} (2022), 1--58] and Pusateri-Soffer [Mem. Amer. Math. Soc. {\bf 299} (2024),  v+107 pp]  to the  critical case. 
		
		As an application, we establish the global well-posedness for cubic NLS with inverse-square potential
		\begin{equation*}
			iu_{t} +\La u = -|u|^2u,\ \ (t,x)\in \mathbb{R}\times\mathbb{R}^{3}.
		\end{equation*}
		for radial initial data $u_0\in H_a^s$ with $s>\frac{11}{13}$ and $a>0$ where $H_a^s$ is the Sobolev space adapted to the Schr\"odinger operator with inverse-square potential. The proof relies on Bourgain's high-low frequency decomposition argument and our bilinear estimates. This is the first result on the low regularity well-posedness for NLS with scaling-critical potential. 
	\end{abstract}
	
	\subjclass[2020]{Primary 35Q55; Secondary 42B37}

	\keywords{Bilinear Strichartz estimate, inverse square potential, Hankel transform, low regularity, Bourgain space.}

	\maketitle\tableofcontents
	
	\section{Introduction}
	\subsection{Statement of main results}
	In this article, we study the bilinear Strichartz estimates for the Schr\"odinger equation with inverse-square potential
	\begin{equation}\label{LS}
		\begin{cases}
			iu_{t}+\La u=0,\\
			u(0,x)=u_0(x), x\in\R^3.
		\end{cases}	
	\end{equation}
	where $ u : \mathbb{R} \times \mathbb{R}^{3} \to \mathbb{C}, \mathcal{L}_{a} = -\Delta + \frac{a}{|x|^{2}}, a\geq0$.
	We consider the self-adjoint operator $\mathcal{L}_a$ as the Friedrichs extension associated with the 
	quadratic form $f\rightarrow \int_{\mathbb{R}^{d}}|\nabla f(x)|^{2}+\frac{a}{|x|^{2}}|f(x)|^{2}dx$ over
	$C_0^\infty(\R^d\setminus\{0\})$. 
	Hardy's inequality then ensures the following equivalence of Sobolev norms
	\begin{align*}
		\|u\|_{\dot H^1(\R^d)}\sim \|u\|_{\dot H_a^1(\R^d)}=\|\sqrt{\mathcal{L}_a} u\|_{L^2}.
	\end{align*}
	When the inverse square potential is attractive ($a<0$), one can additionally impose $a>-\left(\frac{d-2}{2}\right)^2$ to ensure the non-negativity of $\La$, and the equivalence of Sobolev norms also holds in this scenario.
	
	Our main result is the following global-in-time bilinear Strichartz estimate
	\begin{theorem}\label{thm1}
		For $a\geq0$, there holds
		\begin{align}\label{bilinear-thm}
			\big\|e^{it\La}P_N^af\,e^{it\La}P_M^ag\big\|_{L_{t,x}^2(\R\times\R^3)}\leq CMN^{-\frac12}\|f\|_{L^2(\R^3)}\|g\|_{L^2(\R^3)},
		\end{align}
		where $M,N\in2^\Z$ with $M\leq N$ and $P_N^a$ denotes the Littlewood--Paley projection associated to $\La$. The constant may depend on $a$ and the fixed cutoffs, but not on $M,N,f,g$. No radiality or angular regularity is assumed, and no uniformity as $a\downarrow0$ is asserted.
	\end{theorem}

	\begin{remark}
		The estimate holds for arbitrary initial data for every $a\geq0$. When $a=0$, the interaction potential term vanishes identically. We notice that \eqref{bilinear-thm} is sharp in the sense that our estimate matches the bilinear Strichartz estimate in the free case, i.e. $a=0$. Compared to the previous work of \cite{PTV21,BPS}, our bilinear estimate is global-in-time since we have the global-in-time interaction Morawetz estimate, see Zhang-Zheng \cite{ZZ} for more details.  
	\end{remark}

	As an application of \eqref{bilinear-thm}, we adapt the high-low frequency decomposition technique   to our setting and show the low regularity global well-posedness  for cubic NLS with repulsive inverse-square potential in $\R\times\R^3$. 
	\begin{equation}\label{NLS}\tag{$\text{NLS}_a$}
		\left\{
		\begin{aligned}
			& i u_{t} + \mathcal{L}_{a} u =  -|u|^{2} u, \\
			& u(0,x) = u_{0} (x),
		\end{aligned}
		\right. \qquad (t, x) \in \R \times \R^d.
	\end{equation}

	We prove the following global well-posedness for \eqref{NLS} below the energy space $H_a^1$.
	\begin{theorem}\label{thm: nonlinearmain}
		Let $a>0$, the Cauchy problem \eqref{NLS} is {globally}  well-posed for the radial initial data in $H_a^{s}(\R^3)$ with $s>\frac{11}{13}$ in the sense that
		\begin{align*}
			u(t)=e^{it\La}u_0+v(t),\,\,\|v(t)\|_{H_a^1}\lesssim(1+|t|)^{\frac{2(1-s)}{13s-11}+\varepsilon},
		\end{align*}
		where $\varepsilon>0$.	
	\end{theorem}
	\begin{remark}
		For radial initial data, our result holds for $a>-\frac{1}{4}$. Moreover, using the radial Sobolev embedding and Morawetz estimates one can show the scattering for $s>\frac{5}{7}$. We refer to \cite{Bourgain} for the details.
	\end{remark}
	
	\subsection{Background and motivations}
	Let us briefly review the progress on the bilinear Strichartz estimates. For the free NLS, the following bilinear Strichartz estimate holds
	\begin{align}\label{bilinear-intro}
		\big\|e^{it\Delta}P_Nfe^{it\Delta}P_Mg\big\|_{L_{t,x}^2(\R\times\R^d)}\lesssim \frac{M^\frac{d-1}{2}}{N^\frac12}\|f\|_{L^2(\R^d)}\|g\|_{L^2(\R^d)}.
	\end{align}
	This was proved for $d=2$ in \cite{Bourgain-IMRN} and later was generalized to higher dimensions in $d\geq3$ in \cite{Bourgain} and \cite{Visan-Duke} by   using the Fourier transform and Plancherel theorem. Inspired by the interaction Morawetz estimate in \cite{CKSTT-CPAM}, Planchon-Vega \cite{PV} developed the bilinear virial argument to reprove the bilinear Strichartz estimate on $\R^d$ and prove the analogue of \eqref{bilinear-intro} on exterior domain. Their argument is different from \cite{Bourgain} since they do not rely on the convolution structure and Fourier transform. Later, Planchon \cite{P14} applied this argument to further  show the local-in-time bilinear Strichartz estimates on a compact domain $\Omega\subset \R^3$ with additional loss of derivatives
	\begin{align*}
		\big\|e^{it\Delta_\Omega}P_Nfe^{it\Delta_{\Omega}}P_Mg\big\|_{L_{t,x}^2(I\times\Omega)}\lesssim \frac{M}{N^\frac{1}{2}}\|f\|_{L^2(\Omega)}\|g\|_{L^2(\Omega)}
	\end{align*}  
	for $1\leq M\leq N$ and $|I|\leq c N^{-1}$ with $c>0$. 
	Meanwhile, using the bilinear oscillatory integral estimate and the WKB parametrix for $e^{it\Delta_g}$, Hani \cite{Hani} showed the similar estimate holds  for two-dimensional compact manifold without boundary on time interval $I$ with length smaller than $N^{-1}$,
	\begin{align*}
		\big\|e^{it\Delta_g}P_Nfe^{it\Delta_{g}}P_Mg\big\|_{L_{t,x}^2(I\times M)}\lesssim \frac{M}{N^\frac{1}{2}}\|f\|_{L^2(M)}\|g\|_{L^2(M)}.
	\end{align*}
	For the Schr\"odinger equation with trapping potential $H=-\Delta+V$. A typical example is that $V(x)=|x|^2$, Planchon-Tzvetkov-Visciglia \cite{PTV21} showed the sharp bilinear Strichartz estimates in $d=2$. Later, Burq-Poiret-Thomann \cite{BPT} extended it to higher-dimensional case with $\varepsilon$-loss of derivatives.

	For the Schr\"odinger operator with decaying potential, the first result was proved in Camps \cite{Camps}. They showed the bilinear Strichartz estimates for the short range class. The assumption on short range class ensures the existence of distorted Fourier transform and the $L^p$-boundedness of the associated wave operators, see \cite{Yajima}. To prove the bilinear estimate, they viewed the potential term $Vu$ as the inhomogeneous term. Using Bourgain's bilinear estimate as a black-box and the global-in-time local smoothing estimate and further assume that $V$ is smooth, they obtain the desired estimate.  The extra regularity of $V$ can ensure the pseudo-differential calculus.
	Later, Pusateri-Soffer \cite{Soffer-Pusateri} showed the bilinear Strichartz estimates under the assumption of $V$:
	\begin{align*}
		V(x)\in H^{s_1},\,\,\int_{\R^3}(1+|x|)^{s_2+10}|\nabla_x^\alpha V(x)|\,\dd x<\infty,\,0\leq|\alpha|\leq s_2+10
	\end{align*}
	where $s_1=2000$ and $s_2=200$. However, both \cite{Camps} and \cite{Soffer-Pusateri} fail to cover the inverse-square potential, which has the same scaling as the Laplacian and can not be viewed as a perturbation.
	\begin{remark}
		Compared to our work, the assumption of $V$ in \cite{Camps} makes it possible to show the equivalence of Sobolev norms in a wide range of $p$ and $s$. For $V(x)=a\langle x\rangle^{-2-\varepsilon}$, there is no  restriction on the parameter $a$. In our setting, the range of 
		$p$ for which the equivalence of Sobolev norms holds is too narrow to allow the flexible use of Strichartz estimates. 
	\end{remark}
	\subsection{Nonlinear Schr\"odinger equation with inverse-square potentials}
	In this part, we review the progresses on the NLS with inverse-square potentials $V(x)=a|x|^{-2}$ with $a>-\frac14$. 
	\begin{equation}\label{GNLS}
		\left\{
		\begin{aligned}
			& i u_{t} + \mathcal{L}_{a} u =  \mu|u|^{2} u, \\
			& u(0,x) = u_{0} (x),
		\end{aligned}
		\right. \qquad (t, x) \in \R \times \R^3,
	\end{equation}
	where $\mu=\pm1$.

	Formally, the solutions of \eqref{GNLS} conserve the mass and energy by
	\begin{align*}
		& M(u(t)) := \int_{\R^d} |u(t,x)|^2 \,{\rm d}x, \\
		& E_a(u(t)) := \int_{\R^d} \tfrac12|\nabla u(t,x)|^2 + \tfrac{a}{2|x|^2} |u(t,x)|^2 - \tfrac{\mu}{4} |u(t,x)|^{4} \,{\rm d}x.
	\end{align*}
	
	Note that, if $a=0$, then \eqref{GNLS} reduces to the free nonlinear Schr\"odinger equation. For convenience, we start by recalling
	the existing results in $\R^d$, which has been extensively studied in recent years.
	\begin{equation}\label{nls0}\tag{$\text{NLS}_0$}
		iu_{t} + \Delta u = \mu|u|^{2}u.
	\end{equation}
	The solution of the equation \eqref{nls0} is invariant under the rescaling
	\begin{equation}\label{scaling}
		u(t,x)\to u^{\lambda}(t,x) := \lambda u(\lambda^{2}t,\lambda x),
	\end{equation}
	which identifies $\dot H_x^{\frac12}(\R^3)$ as the scaling-critical space of initial data. However, compared to \eqref{nls0}, the equation \eqref{GNLS} with $a\neq 0$ is not space-translation invariant.

	Let us briefly recall the results on cubic NLS with or without potential \eqref{NLS} and \eqref{nls0}. For the free NLS \eqref{nls0} with $\mu=-1$, the global well-posedness and scattering in $\dot H^1$ was proved by  Ginibre-Velo \cite{GV} using the Lin-Strauss type Morawetz estimate. For the focusing case $\mu=1$, using the concentration-compactness argument developed in Kenig-Merle \cite{KM}, Duyckaerts-Holmer-Roudenko \cite{DHR} proved the scattering  in energy space where the mass-energy is below that of the ground state.

	Since there is no conservation law in low regularity space $0<s<1$, the lack of conservation laws at the inter-critical regularity makes it difficult to establish the global well-posedness in $H^s$ with $0<s<1$.  
	Bourgain \cite{Bourgain} first exploited the high-low frequency decomposition argument to show the global well-posedness for this model in $d=3$ for $s>\frac{11}{13} $. Later, Colliander-Keel-Staffilani-Takaoka-Tao \cite{CKSTT-CPAM} introduced the so-called $I$-method and the interaction Morawetz estimate to show the global well-posedness and scattering for $s>\frac45$. Later, Dodson \cite{Dodson} utilized a linear-nonlinear decomposition to improve the result to $s>\frac57$.  Inspired by the work of mass-critical problem in \cite{Dodson-JAMS}, Dodson \cite{Dodson-Camb} utilized the $I$-method as well as the long-time Strichartz estimate to show the global well-posedness and scattering in $H^s$ with  $s>\frac12:=s_c$ for radial initial data which was almost sharp up to the endpoint. Since the $I$-method is actually a subcritical method, we can only prove the result in a subcritical space. For the critical space, it is not valid anymore.    Very recently, Dodson \cite{Dodson-RMI} showed the scattering for $3D$ cubic NLS for $u_0\in\dot{H}^\frac12\cap \dot W^{\frac{7}{6},\frac{11}{7}}(\R^3)$ and  Shen-Wu \cite{SW} further improved the result to $u_0\in \dot{H}^\frac12\cap\dot{W}^{s,1}$ for $s>\frac{12}{13}$. 
	
	Now, we focus on the case that $a\neq0$. For the Schr\"odinger equation with inverse-square potential, the Strichartz estimate  was established in \cite{BPS}. For the defocusing case, Zhang-Zheng \cite{ZZ} showed the scattering in energy space via the new interaction Morawetz estimate. Later, Killip-Miao-Visan-Zhang-Zheng \cite{KMVZZ1} established several harmonic analysis tools associated with $\la$, including the Littlewood-Paley theory and the Mikhlin multiplier theorem. Using these tools, they \cite{KMVZZ} further showed the scattering for the  quintic case in $d=3$. In the same paper, they also discussed the extension to focusing case for radial data. For higher dimensions, we refer to \cite{YZ}. For \eqref{NLS} with energy sub-critical nonlinearities and $\mu=-1$, Killip-Visan-Murphy-Zheng \cite{KMVZ} and Lu-Miao-Murphy \cite{LMM} showed the scattering in energy space where the mass and energy of initial data are below that of ground state based on the concentration-compactness argument. Later, Zheng \cite{Zheng} gave a new proof of scattering using the Dodson-Murphy's argument \cite{DM-PAMS}.   For the mass-energy threshold case, we refer to \cite{MMZ-IUMJ,YZZ} for more details.

	\subsection{Outline of the proof and organization}For the bilinear Strichartz estimate, we adapt the approach in \cite{PTV21}. Our novel contribution is extending such local bilinear estimates in \cite{PTV21} to global one,
	which heavily relies on the Fourier localization, the equivalence of Sobolev norms Theorem \eqref{equivalence} and the interaction Morawetz estimate. Indeed, we first show the following two bilinear estimate holds:
	\begin{equation}\label{eq-H^1 estimate2-in}
		\int_0^T \int_{\R^3}\big| u_N(x)\nabla_x\bar{v}_M(x) \big|^2{\rm d}x{\rm d}t \leq CM^2N\|\Delta_N\varphi_1\|_{L^2(\R^3)}^2\|\Delta_M\varphi_2\|_{L^2(\R^3)}^2,
	\end{equation}
	\begin{equation}\label{eq-H^1 estimate1-in}
		\int_0^T \int_{\R^3}\big| \bar{v}_M(x)\nabla_xu_N(x) \big|^2{\rm d}x{\rm d}t \leq CM^2N\|\Delta_N\varphi_1\|_{L^2(\R^3)}^2\|\Delta_M\varphi_2\|_{L^2(\R^3)}^2.
	\end{equation}
	By the equivalence of Sobolev norms, it implies that 
	\begin{equation}\label{eq-A^1 estimate-2}
		\int_\R\|e^{it\mathcal{L}_a}\Delta_N\varphi_1 e^{it\mathcal{L}_a}\Delta_M\varphi_2\|_{\dot H_a^1(\mathbb{R}^3)}^2{\rm d}t \leq CM^2N\|\Delta_N\varphi_1\|_{L^2(\R^3)}^2\|\Delta_M\varphi_2\|_{L^2(\R^3)}^2.
	\end{equation}
	Then performing the frequency decomposition to $uv$ and the above estimate in $\dot H_a^1$ space, we can complete the proof. The last step is to derive the $L^2$ bilinear Strichartz estimate using the above bilinear estimate in $\dot H_a^1$ space. By the Littlewood-Paley decomposition, we can divide the bilinear term into low frequency and high frequency parts. The proof of high frequency part will follow the strategy in \cite{PTV21}, where the interaction Morawetz estimate and the localization lemma are used. {For the standard Laplacian, by the classical elliptic regularity theory, the localization lemma reads as}
	\begin{equation}\label{fml-point-bound}
		|\phi(x)|^2 \leq C\int_{|x-y|<\lambda^{-1}}(\lambda^{-1}\big| \Delta\phi \big|^2 + C\lambda^3 \big| \phi \big|^2){\rm d}y, ~\forall x\in\R^3.
	\end{equation}
	Thanks to Miao-Su-Zheng \cite{MSZ}, for functions $\phi\in C_0^\infty(\R^3\setminus\{0\})$ and $a>\frac{15}{4}$, one has 
	\begin{equation}\label{fml-L2-control}
		\int_{\R^3} |\Delta \phi|^2\mathrm{d}x \lesssim \int_{\R^3} |\L_a \phi|^2\mathrm{d}x.
	\end{equation}
	The comparison \eqref{fml-L2-control} is not needed in the proof below. Instead, Kato's inequality and the mean-value inequality give the analogue of \eqref{fml-point-bound} with $-\Delta$ replaced by $\La$ for every $a\geq0$. The localization and interaction Morawetz steps that follow this pointwise estimate are retained in full. 
	
	Next, we turn to the low frequency part. As before, spectral localization allows us to write $u_N=N^{-2}\La\widetilde u_N$. A direct product expansion would introduce $|x|^{-2}\widetilde u_Nv_M$ and $\nabla\widetilde u_N\cdot\nabla v_M$. We instead use
	\[
	 v_M\La\widetilde u_N-\widetilde u_N\La v_M
	 =\operatorname{div}(\widetilde u_N\nabla v_M-v_M\nabla\widetilde u_N),
	\]
	in which the potential cancels. The already established bilinear gradient estimates and an $L^2$ bound for the spectrally localized divergence then control the low frequency term for every $a\geq0$. The original high frequency argument is unchanged in substance. The revised proof treats arbitrary initial data directly and does not require a separate angular case.
	
	The paper is organized as follows. In Section 2, we present some notations, and recall some facts from harmonic analysis, such as heat kernel bounds, the equivalence of Sobolev spaces and some related estimates for Schr\"odinger group $e^{it\La}$.
	In Section 3, we introduce the interaction Morawetz quantity and compute the  virial identities.
	In Section 4, we prove Theorem \ref{thm1} by extending the elliptic estimate to our setting involving the decaying potential. In Section 5, we show the global well-posedness for \eqref{NLS} in low regularity space.
	
	\section{Preliminaries}\label{section preliminaries}
	In this section, we will introduce some notations and several fundamental lemmas needed in this paper.
	First, the notation $A\lesssim B$ means that ${A}\leqslant{CB}$ for some constants $ C > 0 $.
	Likewise, if ${A}\lesssim{B}\lesssim{A}$, we say that ${A}\thicksim{B}$. We will also use the notations $A \vee B :=\max\{A,B\}$.
	We use $L^{r}_{x}(\mathbb{R}^d)$ to denote the Lebesgue space of functions $f:\mathbb{R}^d\rightarrow{\mathbb{C}}$ with norm
	$$
	\|f\|_{L^{r}_x(\R^d)}:= \Big( \int_{\mathbb{R}^d}|f(x)|^{r}{\rm d}x \Big)^{\frac{1}{r}}
	$$
	finite, with the usual modifications when $r=\infty$. We also use the space-time Lebesgue spaces $L^{q}_{t}L^{r}_{x}$ which are equipped with the norm
	$$
	\|f\|_{L^{q}_{t}L^{r}_{x}(I\times\R^d)}:=\Big(\int_{I}\|f\|^{q}_{L^r_{x}(\R^d)}{\rm d}t\Big)^{\frac{1}{q}}
	$$
	for any space-time slab $I\times{\mathbb{R}^d}$. When $q=r$, we abbreviate $L^{q}_{t}L^{r}_{x}$ by $L^{q}_{t,x}$.
	
	The Fourier transform on $\R^d$ is defined by
	\begin{align*}
		\hat{f}(\xi):=(2\pi)^{-\frac{d}{2}}\int_{\R^d}e^{-ix\cdot\xi}f(x)\,{\rm d}x,
	\end{align*}
	giving rise to the fractional differentiation operator $|\nabla|^s$, defined by
	\begin{align*}
		|\nabla|^sf(x) := \mathcal{F}^{-1}_{\xi}\big( |\xi|^s\hat{f}(\xi) \big)(x).
	\end{align*}
	In this way, one can define the homogeneous Sobolev space by
	\begin{align*}
		\|f\|_{\dot{W}_x^{s,p}(\R^d)} := \big\| |\nabla|^sf \big\|_{L_x^p(\R^d)}.
	\end{align*}
	
	
	
	\subsection{Harmonic analysis adapted to $\mathcal{L}_a$}
	In this subsection, we introduce some harmonic analysis tools adapted to the operator $\mathcal{L}_a$.
	
	Firstly, for $1<r<\infty$, we write $\dot{H}_{a}^{s, r}(\mathbb{R}^{d})$ and $H_{a}^{s, r}(\mathbb{R}^{d})$ to denote the homogeneous and inhomogeneous Sobolev spaces associated with the operator $\mathcal{L}_{a}$ respectively,
	\[
	\begin{aligned}
		\|f\|_{\dot{H}_{a}^{s, r}(\mathbb{R}^{d})} = \big\| (\mathcal{L}_{a})^{\frac{s}{2}}f \big\|_{L^{r}(\mathbb{R}^{d})},
		\|f\|_{H_{a}^{s, r}(\mathbb{R}^{d})} = \big\| (1+\mathcal{L}_{a})^{\frac{s}{2}}f \big\|_{L^{r}(\mathbb{R}^{d})}.
	\end{aligned}
	\]
	When $r=2$, we simply write $\dot{H}_{a}^{s}(\mathbb{R}^{d})=\dot{H}_{a}^{s, r}(\mathbb{R}^{d})$ and $H_{a}^{s}(\mathbb{R}^{d})=H_{a}^{s, r}(\mathbb{R}^{d})$.
	
	Then by the sharp Hardy inequality, one has
	\begin{equation}\label{iso}
		\big\| \sqrt{\mathcal{L}_a}f \big\|_{L_x^2(\R^d)}^2 \sim \big\| \nabla f \big\|_{L_x^2(\R^d)}^2 \quad
		\textrm{for} ~ a > -\frac{(d-2)^2}4,
	\end{equation}
	thus, the operator $\mathcal{L}_a$ is positive precisely for $ a \ge -\frac{(d-2)^2}4 $. Hereafter, we denote
	\begin{equation}
		\sigma := \frac{d-2}2 - \sqrt{\frac{(d-2)^2}4 + a}.
		\label{sigma}
	\end{equation}
	
	Estimates on the heat kernel associated to the operator $\mathcal{L}_{a}$ were found in \cite{MS}.
	
	\begin{theorem}[Heat kernel bounds, \cite{MS}]\label{heat}
		Fix $ d \ge 3 $ and $ a \ge -\frac{(d-2)^2}4 $. For all $t>0$ and all $x, y\in\mathbb{R}^{d}\backslash\{0\}$, there exist positive constants $C_{1}, C_{2}$ and $c_{1}, c_{2}$ such that
		\[
		C_{1}\big( 1\vee\frac{\sqrt{t}}{|x|} \big)^{\sigma}\big( 1\vee\frac{\sqrt{t}}{|y|} \big)^{\sigma}t^{-\frac{d}{2}}e^{-\frac{|x-y|^{2}}{c_{1}t}} \le e^{-t\mathcal{L}_{a}}(x, y) \le  C_{2} \big( 1\vee\frac{\sqrt{t}}{|x|} \big)^{\sigma}\big( 1\vee\frac{\sqrt{t}}{|y|} \big)^{\sigma}t^{-\frac{d}{2}}e^{-\frac{|x-y|^{2}}{c_{2}t}}.
		\]
	\end{theorem}
	
	As a consequence,  in \cite{KMVZZ1} obtained the  equivalence of Sobolev spaces. Later, Miao-Su-Zheng \cite{MSZ} proved the wider range of $s$ and $p$ which include $s=2$.
	
	\begin{theorem}[Equivalence of Sobolev spaces, \cite{KMVZZ1,MSZ}]\label{equivalence}
		Let $\nu_0 = \sqrt{\frac{1}{4}+a}$. Fix $ d\ge 3 $, $ a > -\frac{(d-2)^2}4 $, $-d<s<2+2\nu_0$ and $s\ne 0$.
		If $1<p<\infty$ satisfies
		\begin{equation*}
			\max\Big\{0,\frac{s+\sigma}{d}\Big\}<\frac{1}{p}<\min\Big\{1, \frac{d-\sigma}{d}\Big\}
		\end{equation*}
		then
		\begin{equation}\label{fml-pre-equiv-1}
			\big\| (-\Delta)^{\frac{s}{2}}f \big\|_{L^{p}_x(\R^d)} \le C(p,s) \big\| \mathcal{L}_{a}^{\frac{s}{2}}f \big\|_{L^{p}_x(\R^d)}\ \ \ \text{for all}\ f\in C_{c}^{\infty}(\mathbb{R}^{d}).
		\end{equation}
		Moreover, assume that
		\begin{equation*}
			\max\Big\{\frac{s}{d}, \frac{\sigma}{d}\Big\}<\frac{1}{p}<\min\Big\{1, \frac{d-\sigma}{d}\Big\},
		\end{equation*}
		which ensures already that $1<p<\infty$, then
		\begin{equation}\label{fml-pre-equiv-2}
			\big\| \mathcal{L}_{a}^{\frac{s}{2}}f \big\|_{L^{p}_x(\R^d)}\le C(p,s)\big\| (-\Delta)^{\frac{s}{2}}f \big\|_{L^{p}_x(\R^d)}\ \ \ \text{for all}\ f\in C_{c}^{\infty}(\mathbb{R}^{d}).
		\end{equation}
	\end{theorem}
	As a consequence of equivalence of homogeneous Sobolev norms, we have the following results.

	\begin{lemma}[Fractional product rule]\label{product}
		Fix $ d \ge 3 $ and $ a \ge -\frac{(d-2)^2}4 $. Then for all $f, g\in C^{\infty}_{c}(\mathbb{R}^{d})$,
		\[
		\big\| \sqrt{\mathcal{L}_{a}}(fg) \big\|_{L^{p}_x(\mathbb{R}^{d})}\le C\Big( \big\| \sqrt{\mathcal{L}_{a}}f \big\|_{L^{p_{1}}_x(\mathbb{R}^{d})} \|g\|_{L^{p_{2}}_x(\mathbb{R}^{d})} + \|f\|_{L^{q_{1}}_x(\mathbb{R}^{d})}\big\| \sqrt{\mathcal{L}_{a}}g \big\|_{L^{q_{2}}_x(\mathbb{R}^{d})} \Big)
		\]
		for any exponents satisfying $ \max \Big\{ \frac{1+\sigma}d, \frac{1}d, \frac{\sigma}d \Big\} < \frac1p, \frac1{p_1}, \frac1{p_2}, \frac1{q_1}, \frac1{q_2} < \min \Big\{1, \frac{d-\sigma}{d} \Big\} $ and $\frac{1}{p}=\frac{1}{p_{1}}+\frac{1}{p_{2}}=\frac{1}{q_{1}}+\frac{1}{q_{2}}$.
	\end{lemma}
	Next, we will introduce the Littlewood-Paley theory associated to the operator $\La$. Let $\varphi:[0, \infty)\to [0, 1]$ be a smooth positive function obeying
	\begin{align*}
		\varphi(\lambda)=
		\begin{cases}
			1,&\ \ \text{for}\ \ 0\le\lambda\le 1,\\
			0,&\ \ \text{for}\ \ \lambda\ge 2.
		\end{cases}
	\end{align*}
	For each dyadic number $N\in 2^{\mathbb{Z}}$, we define
	\[
	\varphi_{N}(\lambda):=\varphi(\lambda/N)\ \ \text{and}\ \ \phi_{N}(\lambda):=\varphi_{N}(\lambda)-\varphi_{N/2}(\lambda).
	\]
	It is clear that $\{\phi_{N}(\lambda)\}_{N\in 2^{\mathbb{Z}}}$ is a partition of unity for $(0, \infty)$. We define the Littlewood-Paley projections via
	\[
	\begin{aligned}
		&f_{\le N}:=P^{a}_{\le N}f:=\varphi_{N}\big( \sqrt{\mathcal{L}_{a}} \big), \\ &f_{N}:=P^{a}_{N}f:=\phi_{N}\big( \sqrt{\mathcal{L}_{a}} \big), \\
		\ \ \ &f_{>N}:=P^{a}_{>N}f:=\big( I-P^{a}_{\le N} \big)f.
	\end{aligned}
	\]
	We will also make use of Littlewood-Paley projections  via  heat kernel:
	\[
	\tilde{P}^{a}_{N}:=e^{-\mathcal{L}_{a}/N^{2}}-e^{-4\mathcal{L}_{a}/N^{2}}.
	\]
	
	%
	Next we recall the following Bernstein and square function estimates from \cite{KMVZZ1}.
	\begin{lemma}[Bernstein estimates, \cite{KMVZZ1}]\label{Bernstein}
		Fix $ d \ge 3 $.
		For $1<p\le q\le\infty$ when $a\ge 0$ or $r_{0}<p\le q<r'_{0}:=\frac{d}{\sigma}$ when $ -\frac{(d-2)^2}4 \le a < 0 $, the following hold:
		\begin{enumerate}
			\item The operators $P^{a}_{\le N}, P^{a}_{N}$ and $\tilde{P}^{a}_{N}$ are bounded on $L^{p}_x(\R^d)$.
			\item The operators $P^{a}_{\le N}, P^{a}_{N}$ and $\tilde{P}^{a}_{N}$ map $L^{p}_x(\R^d)$ to $L^{q}_x(\R^d)$ with norm $O(N^{\frac{d}{p}-\frac{d}{q}})$.
			\item For any $s\in\mathbb{R}$,
			\[
			N^{s}\|P^{a}_{N}f\|_{L^{p}_{x}(\R^d)}\sim \big\|(\mathcal{L}_{a})^{\frac{s}{2}} P^{a}_{N} f \big\|_{L^{p}_{x}(\R^d)}\ \ \text{and}\ \ N^{s}\|\tilde{P}^{a}_{N}f\|_{L^{p}_{x}(\R^d)}\sim \big\|(\mathcal{L}_{a})^{\frac{s}{2}} \tilde{P}^{a}_{N} f\big\|_{L^{p}_{x}(\R^d)}.
			\]
		\end{enumerate}
	\end{lemma}
	
	%
	%
	
	Burq, Planchon, Stalker, and Tahvildar-Zadeh in \cite{BPS} proved the Strichartz estimates for the propagator $e^{-it\mathcal{L}_a}$ for $ d \ge 2 $.
	
	\begin{theorem}[Strichartz estimates, \cite{BPS}]\label{strichartz}
		Fix $ d \ge 2 $ and $ a > -\frac{(d-2)^2}4 $. The solution $ u $ to
		\[
		iu_{t} = \mathcal{L}_{a}u + F
		\]
		on an interval $I\ni t_{0}$ satisfies
		\begin{equation*}
			\| u \|_{L^{q}_{t}L^{r}_{x}(I\times\mathbb{R}^{d})} \le C\big( \| u(t_{0}) \|_{L^{2}(\mathbb{R}^{d})} + \|F\|_{L^{\tilde{q}'}_{t}L^{\tilde{r}'}_{x}(I\times\mathbb{R}^{d})}),
		\end{equation*}
		whenever $\frac{2}{q}+\frac{d}{r} = \frac{2}{\tilde{q}} + \frac{d}{\tilde{r}} = \frac{d}{2}$, $2\le q, \tilde{q}\le\infty$, and $q\neq \tilde{q}$. Moreover, we call such pairs $(q,r)$ and $(\tilde{q}, \tilde{r})$ admissible pairs.
	\end{theorem}
	Next, we need a weighted $L^2$ estimate for operator $e^{it\La}$. Before presenting the result, we first introduce some notations borrowed from \cite{BPS}.
	Let $d\ge3$ and $\lambda(d)=\frac{d-2}{2}$. We define
	\begin{equation*}
		\mu_{b}(d,a) = \sqrt{(\lambda(d)+b)^2+a}.
	\end{equation*}
	Next, we denote $\Omega^s$ by 
	\begin{equation*}
		(\Omega^s f)(x) = |x|^s f(x).
	\end{equation*}
	\begin{theorem}[Local smoothing, \cite{BPS}]\label{thm-locals}
		Let  $b\ge 0$ and $0<\alpha<\frac{1}{4}+\frac{1}{2}\mu_{b}$. Then,  there exists $C>0$ such that
		\begin{equation*}
			\|\Omega^{-\frac{1}{2}-2\alpha} e^{it\La}u_0\|_{L^2_{t,x}(\R\times \R^d)} \le C \|\La^{\alpha-\frac{1}{4}} u_0\|_{L^2_x(\R^d)}
		\end{equation*}
		for initial data $u_0\in L_{>b}^2(\R^d)$, where $L_{>b}^2$ is the subspace of $L^2$ consisting of functions that are orthogonal to all spherical harmonics of degree
		less than $b$.
	\end{theorem}
	\begin{remark}
		In our paper, we focus on $d=3$. Therefore, taking $b=1$ and thus we can take $\alpha=\frac{3}{4}$, i.e.
		\begin{equation}\label{fml-weighedL22}
			\||x|^{-2}e^{it\La}u_0\|_{L^2_{t,x}(\R\times\R^d)} \lesssim \|\La^{\frac{1}{2}}u_0\|_{L^2_x(\R^d)}
		\end{equation}
		for $u_0\in L_{>1}^2(\R^3)$.
	\end{remark}
	
	We also need the dispersive estimate for Schr\"odinger operator with repulsive inverse square potential $a>0$, which was proved in Fanelli-Felli-Fontelos-Primo \cite{FFFP}.
	\begin{lemma}Let $a\geq0$, then the following dispersive estimate holds
		\begin{align*}
			\big\|e^{-it\mathcal{L}_a}f\|_{L_x^\infty(\R^3)}\leqslant C|t|^{-\frac32}\|f\|_{L_x^1(\R^3)}.
		\end{align*}
	\end{lemma}
	%
	\subsection{Distorted plane waves}
	In this part, we discuss the distorted plane wave for the operator $$\mathcal{L}_a=-\Delta+a|x|^{-2}.$$
	The associated eigenfunction is the distorted plane wave $\phi(x,\xi)$, which may be described as the solution to
	\begin{align*}
		\mathcal{L}_a\phi(x,\xi)=|\xi|^2\phi(x,\xi).
	\end{align*}
	In particular, when $a=0$, then the eigenfunction degenerates to the classical plane wave $\phi=ce^{-ix\cdot\xi}$. For $a\neq0$, the distorted eigenfunctions can be formulated by
	\begin{align*}
		\phi(x,\xi)&=(|x|\cdot |\xi|)^{-\frac{1}{2}}\sum_{k=1}^\infty i^{\beta_k}J_{\beta_k}(|x|\cdot |\xi|)\overline{\psi_k(\frac{x}{|x|})}\psi_k(\frac{\xi}{|\xi|})\\
		&=(|x|\cdot |\xi|)^{-\frac{1}{2}}\sum_{\ell=0}^\infty i^{\alpha_\ell}J_{\alpha_\ell}(|x|\cdot |\xi|)\sum_{m=1}^{m_\ell}\overline{Y_{\ell,m}(\frac{x}{|x|})}Y_{\ell,m}(\frac{\xi}{|\xi|}),
	\end{align*}
	where $\{Y_{\ell,m}\}_{m=1}^{m_\ell}$ is  the orthogonal basis of the space of spherical harmonics of degree $\ell$ on $\Bbb S^{d-1}$, $\alpha_\ell=\sqrt{\frac{1}{4}+\lambda_\ell+a}$ and $\lambda_\ell=(\ell+1)\ell$. Also, we denote by
	\begin{align*}
		\beta_k=\begin{cases}
			(a+\frac{1}{4})^\frac12,&k=1,\\
			(a+\lambda_\ell+\frac14)^\frac12,&\sum_{j=0}^{\ell-1}m_j<k<\sum_{j=0}^{\ell}m_j
		\end{cases}
	\end{align*}
	and
	\begin{align*}
		m_\ell=\frac{\ell!(2\ell+1)!}{\ell!},\quad \{\psi_k\}_{k=1}^\infty=\{Y_{\ell,m}\}_{m=1}^{m_\ell}.
	\end{align*}
	Notice that for free case $a=0$, the Jacobi-Anger expansion for plane waves combining with the property of spherical harmonics imply that
	\begin{align*}
		e^{-ix\cdot\xi}=(2\pi)^\frac32(|x|\cdot|\xi|)^{-\frac12}\sum_{\ell=0}^\infty i^\ell J_{\ell+\frac12}(|x|\cdot|\xi|)\sum_{m=1}^{m_\ell}\overline{Y_{\ell,m}(\frac{x}{|x|})}Y_{\ell,m}(\frac{\xi}{|\xi|})=(2\pi)^{\frac32}i^{-\frac12}\phi(x,\xi).
	\end{align*}
	Now, we can define the distorted Fourier transform of $f$ by
	\begin{align*}
		(\mathcal{F}_af)(\xi):=\int_{\R^3}f(x)\phi(x,\xi)dx.
	\end{align*}
	Similar to the case of $a=0$, we also have the Plancherel theorem
	\begin{lemma}
		There holds
		\begin{align}\label{Inverse}
			f(x)=\int_{\R^d}(\mathcal{F}_af)(\xi)\overline{\phi(x,\xi)}d\xi
		\end{align}
		and
		\begin{align}\label{Parseval}
			\int_{\R^3}f\cdot\overline{g}dx=\int_{\R^3}\mathcal F_af\overline{\mathcal F_ag}d\xi.
		\end{align}
	\end{lemma}
	Using the distorted Fourier transform $\mathcal{F}_af$, one can define the Bourgain space $X_a^{s,b}$ as follows.
	\begin{definition}
		Let $s\geq0$, the Bourgain space is the completion of $\mathcal{S}(\R\times\R^3)$ equipped with the norm
		\begin{align*} \|F\|_{X_a^{s,b}(\R\times\R^3)}:=\big\|e^{-it\La}F\|_{H_t^b(\R,H_a^s(\R^3))}=\left(\int_{\R}\int_{\R^3}(1+|\tau-|\xi|^2|^2)^b(1+|\xi|^2)^s\big|\big[\mathcal{F}_t(\mathcal{F}_af)\big](\tau,\xi)\big|^2\,d\xi\,d\tau\right)^\frac12,
		\end{align*}
		where $\mathcal{F}_t$ denote the Fourier transform with respect to $t$.
		For $T > 0$, we can define the restricted Bourgain space equipped with the norm
		\begin{align*}
			\|u\|_{X_a^{s,b}([-T,T])}:=\big\{\|F\|_{X_a^{s,b}(\R\times\R^3)}|F|_{[-T,T]\times\R^3}=u\big\}.
		\end{align*}
		Throughout this paper, we abbreviate $X_{a}^{s,b}:=X_a^{s,b}(\R\times\R^3)$ and $X_a^{s,b}(I)=X_a^{s,b}(I\times\R^3)$ for short.
	\end{definition}
	\begin{remark}
		Let $W_{\pm}$ be the time-dependent wave operator defined via
		\begin{gather*}
			W_{\pm} = \slim_{t\to {\pm \infty}}e^{it \La}e^{it\Delta}, \\
			W_{\pm}^* = \slim_{t\to {\pm \infty}}e^{-it\Delta }e^{-it\La},
		\end{gather*}
		where $W_{\pm}^*$ is the dual operator of $W_{\pm}$ and where $\operatorname{s-lim}$ indicates the strong limit in $L^2(\R^d)$. By definition, we have
		\begin{align*}
			\|f\|_{X_a^{s,b}}&=\|e^{-it\La}f\|_{H_t^bH_a^s}=\|\langle\La\rangle^{\frac{s}{2}}e^{-it\La}f\|_{L_x^2}\|_{H_t^b}\\
			&=\| W_{\pm}\langle\nabla\rangle^sW_{\pm}^*f\|_{L_x^2}\|_{H_t^b}=\|W^*_{\pm}f\|_{X^{s,b}}.
		\end{align*}
	\end{remark}
	Next, we state some basic properties of Bourgain space.
	\begin{proposition}
		The following statements hold for $s\geq0$ and $b>\frac{1}{2}$.
		\begin{enumerate}
			\item We have the following Sobolev embedding:
			\begin{align*}
				X_{a}^{s,b}\hookrightarrow C(\R,H_a^s(\R^3)).
			\end{align*}
			\item  For $s^\prime \leq s$ and $b^\prime \leq b$ and compact time interval $I$, it holds
			\begin{align*}
				X_a^{s,b}(I) \hookrightarrow X_a^{s^\prime,b^\prime}(I).
			\end{align*}
		\end{enumerate}
	\end{proposition}
	Next we state the homogeneous and inhomogeneous estimates for solutions to the linear Schr\"odinger
	equations in Bourgain spaces, see Burq-G\'erard-Tzvetkov \cite{BGT-Invent} for the detailed proofs.
	\begin{lemma}[Homogeneous Strichartz estimate]
		Let $s, b > 0$ and $f \in H_a^s(\R^3)$,
		, then there exists a constant $C > 0$
		such that
		\begin{align*}
			\big\|\psi(t)e^{it\La}f\big\|_{X_a^{s,b}}\leq C\|f\|_{H_a^{s}(\R^3)},
		\end{align*}
		where $\psi(t)$ is a compactly supported function.
	\end{lemma}
	\begin{lemma}[Inhomogeneous Strichartz estimate]
		Let $0 < b^\prime <
		\frac12$
		and $0 < b < 1 - b^\prime$. Then for all $f \in X_a^{s,-b^\prime}([-T,T])$, the inhomogeneous term $w(t, x) = \int_0^t
		e^{i(t-s)\La} f(s)ds$ satisfies the following estimate
		\begin{align*}
			\big\|\psi(t/T )w(t,x)\big\|_{X_a^{s,b}}\leq  CT^{1-b-b^\prime}\|f\|_{X_a^{s,-b^\prime}([-T,T])},
		\end{align*}
		where $\psi(t)$ is the same function as the above lemma.
	\end{lemma}
	Notice that when replacing the Fourier transform by distorted Fourier transform $\mathcal{F}_a$, the proof will follows immediately with minor modification. Therefore, we omit the details for interesting readers.
	
	\subsection{Spherical harmonics decomposition}
	For any $u\in L^2(\R^3)$, we may write
	\begin{equation*}
		u(x) = \sum_{k=0}^\infty \sum_{\ell =1}^{d(k)} Y_{k,\ell}(r) Y_{k,\ell}(\theta),
	\end{equation*} 
	where $k\in \mathbb{N}\cup \{0\}$, $d(k) = 2k+1$ and $\{Y_{k,\ell}\}_{\ell=0}^{d(k)}$ denote the orthogonal basis of the space of spherical harmonics of degree $k$ on $\mathbb{S}^2$. 
	
	We define a Hankel transform for $u(x)$:
	\begin{equation*}
		(\H_{\nu}u)(\xi) = \int_0^\infty (r\rho)^{-\frac{1}{2}} J_{\nu}(r\rho) f(r\omega)r^2 \dd r,
	\end{equation*}
	where $J_k$ denotes the Bessel function of order $k$:
	\begin{equation*}
		J_k(r) = \frac{r^k}{2^k\Gamma(k+1/2)\Gamma(1/2)}\int_{-1}^1 e^{isr} (1-s^2)^{(2k-1)/2} \dd s,
	\end{equation*}
	with $k>-1/2$, $r>0$ and $\Gamma$ denotes the standard Gamma function.
	
	Let
	\begin{equation*}
		\mu(k) = k+\frac{1}{2},
	\end{equation*}
	one verifies that
	\begin{equation*}
		\H_{\mu(k)} [a_{k,\ell}(r) Y_{k,\ell}(\theta)](\xi) = Y_{k,\ell}(\omega) [\H_{\mu(k)} a_{k,\ell}](\rho).
	\end{equation*}
	Denote that
	\begin{equation*}
		\nu(k) = \sqrt{\mu(k)^2 + a},\,\, a>-\frac{1}{4},
	\end{equation*}
	and
	\begin{equation*}
		A_{\nu(k)} := -\partial^2_r -\frac{2}{r}\partial_{r} + [\nu(k)^2-\frac{1}{4}] r^{-2}.
	\end{equation*}
	
	Then, the Hankel transform has the following properties:
	\begin{enumerate}
		\item $\H_{\nu} = \H_{\nu}^{-1}$,\\
		\item $\H_{\nu}$ is self-adjoint: $\H_{\nu} = \H_{\nu}^*$,\\
		\item $\H_{\nu}$ is $L^2$ isometry, that is $\|\H_{\nu} f\|_{L^2} = \|f\|_{L^2}$,\\
		\item $\H_{\nu(k)}(A_{\nu(k)}f)(\xi) = |\xi|^2 (\H_{\nu(k)}f)(\xi)$ if $f\in L^2$.
	\end{enumerate}

	At the end of this section, we discuss the asymptotic behavior of Bessel functions.
	\begin{lemma}\label{lem-symp-Bessel}
		Let $J_\nu(r)$ denotes the Bessel function of order $\nu$ with $\nu>-\frac{1}{2}$. Then, the following statements hold
		\begin{enumerate}
			\item $|J_\nu(r)| \lesssim z^{\nu}$ if $0<z\ll \sqrt{\nu + 1}$,\\
			\item $|J_\nu(r)| \lesssim z^{-\frac{1}{2}}$ if $z\gg |\nu^2-\frac{1}{4}|$.
		\end{enumerate}
	\end{lemma}

		\section{The proof of bilinear Strichartz estimates}
		In this section, we prove a sharp bilinear Strichartz estimate, i.e. Theorem \ref{thm1}. Our strategy is to utilize the physical space method developed in \cite{P14}. The main ingredient is the interaction Morawetz estimate. Let $u$ satisfies
		\begin{equation}\label{linear eq}
			i u_{t} - \Delta u + \frac{a}{|x|^2}u = 0,
		\end{equation}
		then, we first introduce the interaction Morawetz estimate for $u$.
		In the proof of the frequency-localized estimates below, we assume
		$a\geq0$. Constants may depend on $a$ and the fixed cutoffs, but
		not on the frequencies or the time interval. We retain the original
		range $1\leq M\leq N$ during the estimates and remove the lower
		frequency restriction by scaling at the end.

		\begin{proposition}\label{prop-morawetz}
			Let $u$ and $v$ satisfy \eqref{linear eq}, and let
			$\rho:\R^3\to\R$ be a real-valued $C^4$ function with bounded
			derivatives of orders one through four. Define
			\[
			 I_\rho(t)=\int_{\R^3}\int_{\R^3}
			 |u(t,x)|^2\rho(x-y)|v(t,y)|^2\,\dd x\dd y.
			\]
			Initially assume that the solutions have sufficient regularity and
			decay for the following identities. Then
			\begin{align}
			 \frac{d}{dt}I_\rho(t)
			 ={}&2\Im\iint\nabla\rho(x-y)\cdot
			       \nabla\bar u(x)u(x)|v(y)|^2\,\dd x\dd y\nonumber\\
			 &-2\Im\iint\nabla\rho(x-y)\cdot
			       \nabla\bar v(y)v(y)|u(x)|^2\,\dd x\dd y,
			 \label{eq-interaction-first-full}
			\end{align}
			and
			\begin{align}
			 \frac{d^2}{dt^2}I_\rho(t)
			 ={}&4\iint H_\rho(x-y)(\nabla\bar u(x),\nabla u(x))
			                    |v(y)|^2\,\dd x\dd y\nonumber\\
			 &+4\iint H_\rho(x-y)(\nabla\bar v(y),\nabla v(y))
			                    |u(x)|^2\,\dd x\dd y\nonumber\\
			 &-\iint\Delta\rho(x-y)\Delta_x(|u(x)|^2)
			                    |v(y)|^2\,\dd x\dd y\nonumber\\
			 &-\iint\Delta\rho(x-y)\Delta_y(|v(y)|^2)
			                    |u(x)|^2\,\dd x\dd y\nonumber\\
			 &-8\iint H_\rho(x-y)
			    \big(\Im(\nabla\bar u(x)u(x)),
			         \Im(\nabla\bar v(y)v(y))\big)\,\dd x\dd y\nonumber\\
			 &+4a\iint\left(\frac{x}{|x|^4}-\frac{y}{|y|^4}\right)
			       \cdot\nabla\rho(x-y)|u(x)|^2|v(y)|^2\,\dd x\dd y.
			 \label{eq-interaction-second-full}
			\end{align}
			Here and below, the time variable is suppressed inside spatial
			integrals, all unrestricted double integrals are over
			$\R^3\times\R^3$, and
			\[
			 H_\rho(z)(f,g):=\sum_{j,k=1}^3
			       \partial_j\partial_k\rho(z)f_jg_k,
			 \qquad z\in\R^3,\quad f,g\in\C^3.
			\]
		\end{proposition}

		\begin{proof}
			We retain the single-solution calculation and then compute the
			mixed term separately. Write
			$\Im(\nabla\bar u\,u)$ for the current corresponding to the sign
			convention in \eqref{linear eq}. Directly from the equation,
			\[
			 \partial_t|u|^2=-2\operatorname{div}
			                       \Im(\nabla\bar u\,u).
			\]
			For the classical Morawetz functional
			$I_\rho^u(t):=\int_{\R^3}\rho(x)|u(t,x)|^2\,\dd x$,
			integration by parts therefore gives
			\begin{equation}\label{virial-1}
			 \frac{d}{dt}I_\rho^u(t)
			 =2\Im\int_{\R^3}\nabla\rho(x)\cdot
			                      \nabla\bar u(x)u(x)\,\dd x.
			\end{equation}
			To differentiate once more, use the momentum identity
			\begin{align*}
			 \partial_t\Im(\partial_j\bar u\,u)
			 ={}&\frac12\partial_j\Delta(|u|^2)
			 -2\sum_{k=1}^3\partial_k
			                  \Re(\partial_j\bar u\,\partial_k u)
			 -|u|^2\partial_j\left(\frac{a}{|x|^2}\right).
			\end{align*}
			Multiplying by $2\partial_j\rho$, summing over $j$, and integrating
			by parts yields
			\begin{align}
			 \frac{d^2}{dt^2}I_\rho^u(t)
			 ={}&4\int_{\R^3}H_\rho(x)(\nabla\bar u(x),\nabla u(x))\,\dd x
			 -\int_{\R^3}\Delta\rho(x)\Delta(|u(x)|^2)\,\dd x\nonumber\\
			 &+4a\int_{\R^3}\frac{x}{|x|^4}\cdot\nabla\rho(x)
			                         |u(x)|^2\,\dd x.
			 \label{virial-2}
			\end{align}
			In particular, the coefficient of the potential term is $4a$,
			since $-2\nabla(a|x|^{-2})=4a x|x|^{-4}$.

			Applying \eqref{virial-1} to the $x$ and $y$ variables separately,
			and using $\nabla_y\rho(x-y)=-\nabla\rho(x-y)$, gives
			\begin{align*}
			 \frac{d}{dt}I_\rho(t)
			 ={}&2\Im\iint\nabla\rho(x-y)\cdot
			              \nabla\bar u(x)u(x)|v(y)|^2\,\dd x\dd y\\
			 &-2\Im\iint\nabla\rho(x-y)\cdot
			              \nabla\bar v(y)v(y)|u(x)|^2\,\dd x\dd y.
			\end{align*}
			This proves \eqref{eq-interaction-first-full}.
			Next, the product rule gives the three terms
			\begin{align*}
			 \partial_t^2\big(|u(x)|^2\rho(x-y)|v(y)|^2\big)
			 ={}&\partial_t^2(|u(x)|^2)\rho(x-y)|v(y)|^2\\
			 &+|u(x)|^2\rho(x-y)\partial_t^2(|v(y)|^2)\\
			 &+2\rho(x-y)\partial_t(|u(x)|^2)\partial_t(|v(y)|^2).
			\end{align*}
			By \eqref{virial-2}, the spatial integral of the first two terms is
			\begin{align*}
			 &4\iint H_\rho(x-y)(\nabla\bar u(x),\nabla u(x))
			                                     |v(y)|^2\,\dd x\dd y\\
			 &\quad+4\iint H_\rho(x-y)(\nabla\bar v(y),\nabla v(y))
			                                     |u(x)|^2\,\dd x\dd y\\
			 &\quad-\iint\Delta\rho(x-y)\Delta_x(|u(x)|^2)
			                                     |v(y)|^2\,\dd x\dd y\\
			 &\quad-\iint\Delta\rho(x-y)\Delta_y(|v(y)|^2)
			                                     |u(x)|^2\,\dd x\dd y\\
			 &\quad+4a\iint\left(\frac{x}{|x|^4}-\frac{y}{|y|^4}\right)
			    \cdot\nabla\rho(x-y)|u(x)|^2|v(y)|^2\,\dd x\dd y.
			\end{align*}
			For the mixed term, the continuity identities and two integrations
			by parts give
			\begin{align*}
			 &2\iint\rho(x-y)\partial_t(|u(x)|^2)
			                           \partial_t(|v(y)|^2)\,\dd x\dd y\\
			 &=8\iint\rho(x-y)
			       \operatorname{div}_x\Im(\nabla\bar u(x)u(x))
			       \operatorname{div}_y\Im(\nabla\bar v(y)v(y))\,\dd x\dd y\\
			 &=-8\sum_{j,k=1}^3\iint\partial_j\partial_k\rho(x-y)
			        \Im(\partial_j\bar u(x)u(x))
			        \Im(\partial_k\bar v(y)v(y))\,\dd x\dd y\\
			 &=-8\iint H_\rho(x-y)
			       \big(\Im(\nabla\bar u(x)u(x)),
			            \Im(\nabla\bar v(y)v(y))\big)\,\dd x\dd y.
			\end{align*}
			Combining these terms proves \eqref{eq-interaction-second-full}.

			For completeness, for $a>0$ the calculations can first be made for
			finite spherical-harmonic sums with smooth radial Hankel transforms
			compactly supported away from zero frequency. In degree $\ell$ the
			Friedrichs behavior at the origin is $r^{\beta_\ell}$, where
			\[
			 \beta_\ell=-\frac12+\sqrt{\left(\ell+\frac12\right)^2+a}>0.
			\]
			Performing the action calculation on $|x|,|y|>\varepsilon$, for a
			fixed smooth weight, gives inner boundary terms that vanish as
			$\varepsilon\downarrow0$; the largest possible terms are of order
			$\varepsilon^{2\beta_\ell}$. The force integrals are justified by
			the weighted smoothing estimate used below. The chosen data decay
			rapidly at infinity on compact time intervals. At $a=0$ one instead
			uses Schwartz solutions on all of $\R^3$, with no inner boundary.
			The resulting inequalities extend to $H_a^1$ data by density and
			lower semicontinuity. Thus the estimates below do not require
			finite moments or an angular restriction on the data.
		\end{proof}
		
		Now, we can give the proof   of  our bilinear Strichartz estimates. The first one is in classical version and the second one is the bilinear estimate involving the operator $\La^\frac{s}{2}$ with $s\in[0,1]$.  The proof relies on the interaction Morawetz estimate and the frequency decomposition technique.
		\begin{lemma}\label{le-improved bilinear estimates}
			Let $a\geq0$ and let $1 \leq M \leq N$ be dyadic integers, then there exists $C$ such that
			\begin{equation}\label{eq-A^1 estimate}
				\int_\R\|e^{it\mathcal{L}_a}P_N^a\varphi_1 e^{it\mathcal{L}_a}P_M^a\varphi_2\|_{\dot H_a^1(\mathbb{R}^3)}^2{\rm d}t \leq CM^2N\|P_N^a\varphi_1\|_{L^2(\R^3)}^2\|P_M^a\varphi_2\|_{L^2(\R^3)}^2.
			\end{equation}
		\end{lemma}
		\begin{proof}
			Denote by $u_N:=e^{it\La}P_N^a\varphi_1$ and
			$v_M:=e^{it\La}P_M^a\varphi_2$ for simplicity. Both solve
			\eqref{linear eq}. We retain $\bar v_M$ in the gradient estimates
			and in the positive Hessian expression; complex conjugation does
			not change the absolute values of the gradient products.
			
			By the equivalence of Sobolev norms Theorem \ref{equivalence}, to prove \eqref{eq-A^1 estimate},  it suffices to prove the following two inequalities:
			\begin{equation}\label{eq-H^1 estimate2}
				\int_0^T \int_{\R^3}\big| u_N(x)\nabla_x\bar{v}_M(x) \big|^2{\rm d}x{\rm d}t \leq CM^2N\|P_N^a\varphi_1\|_{L^2(\R^3)}^2\|P_M^a\varphi_2\|_{L^2(\R^3)}^2,
			\end{equation}
			\begin{equation}\label{eq-H^1 estimate1}
				\int_0^T \int_{\R^3}\big| \bar{v}_M(x)\nabla_xu_N(x) \big|^2{\rm d}x{\rm d}t \leq CM^2N\|P_N^a\varphi_1\|_{L^2(\R^3)}^2\|P_M^a\varphi_2\|_{L^2(\R^3)}^2,
			\end{equation}
			For \eqref{eq-H^1 estimate2}, by Cauchy-Schwarz and Bernstein's  inequality, we get
			\begin{align*}
				\big\|u_N\nabla_x\bar{v}_M\big\|_{L_{t,x}^2}\lesssim& \|u_N\|_{L_t^2 L_x^\infty} \|\nabla_x\bar{v}_M\|_{L_t^\infty L_x^2}\\
				\lesssim& N^\frac12\|u_N\|_{L_t^2L_x^6} M \|\bar{v}_M\|_{L_t^\infty L_x^2}\\
				\lesssim&MN^\frac12\|P_N^a\varphi_1\|_{L^2(\R^3)}\|P_M^a\varphi_2\|_{L^2(\R^3)},
			\end{align*}
			which implies  \eqref{eq-H^1 estimate2}.

			
			Next, we only need to prove \eqref{eq-H^1 estimate1}. Before proving this inequality, we introduce a important lemma. This lemma is an analogue of Lemma 4.1 in \cite{PTV21}; its proof below uses only the nonnegativity of the potential.
			\begin{lemma}\label{lemma-ine}
				Let $a\geq0$. For any $\phi\in D(\La)$ and any $\lambda>0$,
				the following pointwise estimate holds for almost every $x\in\R^3$:
				\begin{equation}\label{eq-pointwise-all-a}
				 |\phi(x)|^2\leq C\int_{|x-y|<\lambda^{-1}}
				 \left(\lambda^{-1}|\La\phi(y)|^2
				                     +\lambda^3|\phi(y)|^2\right)\,\dd y.
				\end{equation}
				The constant is independent of $a\geq0$, $\lambda$, and $x$.
			\end{lemma}
			\begin{proof}
				Set $h=\La\phi$. We first verify the local integrability needed
				for Kato's inequality, including at the singularity. Since
				$\phi\in D(\La)\subset H^1(\R^3)$, Hardy's inequality gives
				\[
				 \int_{\R^3}\frac{|\phi(y)|^2}{|y|^2}\,\dd y
				 \leq4\int_{\R^3}|\nabla\phi(y)|^2\,\dd y.
				\]
				For every $R>0$, Cauchy--Schwarz therefore yields
				\begin{align*}
				 \int_{|y|<R}\frac{|\phi(y)|}{|y|^2}\,\dd y
				 &\leq\left(\int_{|y|<R}\frac{|\phi(y)|^2}{|y|^2}\,\dd y\right)^{1/2}
				       \left(\int_{|y|<R}\frac{\dd y}{|y|^2}\right)^{1/2}\\
				 &\leq2(4\pi R)^{1/2}\|\nabla\phi\|_2<\infty.
				\end{align*}
				Thus $|y|^{-2}\phi\in L^1_{\mathrm{loc}}(\R^3)$ and
				$\Delta\phi=a|y|^{-2}\phi-h\in L^1_{\mathrm{loc}}(\R^3)$.
				Follows from the Kato inequality in the distribution sense (see the proof in \cite{BP}),
				\begin{align*}
				 -\Delta|\phi|
				 \leq\Re\left(\frac{\bar\phi}{|\phi|}(-\Delta\phi)\right)=\Re\left(\frac{\bar\phi}{|\phi|}h\right)
				                         -\frac{a}{|y|^2}|\phi|
				 \leq |h|,
				\end{align*}
				where the quotient is set to zero on $\{\phi=0\}$ and the last inequality follows from $a\geq0$.

				Fix $x\in\R^3$, set $R=\lambda^{-1}$ and $B=B(x,R)$, and define
				\[
				 w(z)=\frac1{4\pi}\int_B\frac{|h(y)|}{|z-y|}\,\dd y.
				\]
				Then $w\geq0$ and $-\Delta w=|h|$ in $B$. Hence
				$\Delta(|\phi|-w)\geq0$ in $B$, so that $|\phi|-w$ is
				subharmonic. The mean-value inequality at a Lebesgue point gives
				\begin{align*}
				 |\phi(x)|-w(x)
				 &\leq\frac1{|B|}\int_B\big(|\phi(y)|-w(y)\big)\,\dd y\\
				 &\leq\frac1{|B|}\int_B|\phi(y)|\,\dd y.
				\end{align*}
				The two terms on the resulting right-hand side satisfy
				\begin{align*}
				 \frac1{|B|}\int_B|\phi(y)|\,\dd y
				 &\leq C R^{-3/2}\|\phi\|_{L^2(B)},\\
				 w(x)&\leq\frac1{4\pi}
				       \left(\int_B|h(y)|^2\,\dd y\right)^{1/2}
				       \left(\int_B|x-y|^{-2}\,\dd y\right)^{1/2}\\
				 &\leq C R^{1/2}\|h\|_{L^2(B)}.
				\end{align*}
				Here $\int_{B(x,R)}|x-y|^{-2}\,\dd y=4\pi R$.
				Squaring and substituting $R=\lambda^{-1}$ proves
				\eqref{eq-pointwise-all-a}. In particular, no comparison between
				$\Delta\phi$ and $\La\phi$ in $L^2$ is needed.
			\end{proof}
			
			From the above Lemma \ref{lemma-ine}, \eqref{eq-H^1 estimate1} can estimated by
			\begin{align}\label{eq-derivative for N}
				& \int_0^T\Big( \int_{\R^3}\big| \bar{v}_M(x)\nabla_xu_N(x) \big|^2{\rm d}x \Big){\rm d}t \\
				&\leq   C\int_0^T\Big( \iint_{|x-y|<\frac1M}M^3\big| \bar{v}_M(y)\nabla_x u_N(x) \big|^2
				+ \frac{1}{M}\big| \mathcal{L}_a\bar{v}_M(y)\nabla_x u_N(x) \big|^2{\rm d}x{\rm d}y \Big){\rm d}t.
			\end{align}
			Next we claim that
			\begin{align}\label{eq-claim H}
			 &\int_0^T\iint_{|x-y|<M^{-1}}H_{\tilde\rho}(x-y)
			 \Big(\bar v_M(y)\nabla_xu_N(x)+u_N(x)\nabla_y\bar v_M(y),\nonumber\\
			 &\hspace{48mm}v_M(y)\nabla_x\bar u_N(x)
			                   +\bar u_N(x)\nabla_yv_M(y)\Big)\,\dd x\dd y\dd t\nonumber\\
			 &\quad\leq C\|\nabla\tilde\rho\|_\infty
			 \Big(\|v_M(0)\|_2^2\|u_N(0)\|_2\|u_N(0)\|_{H_a^1}
			       +\|u_N(0)\|_2^2\|v_M(0)\|_2\|v_M(0)\|_{H_a^1}\Big),
			\end{align}
			where $\tilde\rho(z)=\sum_{j=1}^3\rho_M(z_j)$ and the convex
			function $\rho_M$ is specified below. We retain all the steps
			leading from the interaction identity to this localized estimate.

			By Proposition~\ref{prop-morawetz}, Cauchy--Schwarz, and
			conservation of mass and energy, we first have
			\begin{align}\label{eq-one derivative forI}
			 |I_\rho'(t)|
			 &\leq2\|\nabla\rho\|_\infty
			       \Big(\|v(t)\|_2^2\|u(t)\|_2\|\nabla u(t)\|_2
			            +\|u(t)\|_2^2\|v(t)\|_2\|\nabla v(t)\|_2\Big)\nonumber\\
			 &\lesssim\|\nabla\rho\|_\infty
			       \Big(\|v(0)\|_2^2\|u(0)\|_2\|u(0)\|_{H_a^1}
			            +\|u(0)\|_2^2\|v(0)\|_2\|v(0)\|_{H_a^1}\Big).
			\end{align}
			Here $\|\nabla u\|_2\leq\|\La^{1/2}u\|_2$ follows directly
			from $a\geq0$.

			We next rewrite the second derivative as in the original
			interaction argument. To justify the rewriting explicitly,
			expansion of the Hessian form gives
			\begin{align*}
			 &4H_\rho(x-y)\Big(\bar v\nabla u+u\nabla\bar v,
			                              v\nabla\bar u+\bar u\nabla v\Big)\\
			 &=4|v|^2H_\rho(x-y)(\nabla\bar u,\nabla u)
			   +4|u|^2H_\rho(x-y)(\nabla\bar v,\nabla v)\\
			 &\quad+2H_\rho(x-y)(\nabla|u|^2,\nabla|v|^2)\\
			 &\quad-8H_\rho(x-y)
			       \big(\Im(\nabla\bar u\,u),\Im(\nabla\bar v\,v)\big).
			\end{align*}
			The density term satisfies, by integration in $x$ and then in $y$,
			\begin{align*}
			 &\iint H_\rho(x-y)(\nabla_x|u(x)|^2,\nabla_y|v(y)|^2)\,\dd x\dd y\\
			 &=\sum_{j,k=1}^3\iint\partial_j\partial_k\rho(x-y)
			             \partial_{x_j}|u(x)|^2\partial_{y_k}|v(y)|^2\,\dd x\dd y=-\iint\Delta^2\rho(x-y)|u(x)|^2|v(y)|^2\,\dd x\dd y.
			\end{align*}
			Consequently, twice this integral is exactly the sum of the two
			terms involving $\Delta\rho$ in
			\eqref{eq-interaction-second-full}. Thus
			\begin{align}\label{eq-interaction-square-full}
			 I_\rho''(t)
			 ={}&4\iint H_\rho(x-y)
			 \Big(\bar v(y)\nabla u(x)+u(x)\nabla\bar v(y),\nonumber\\
			 &\hspace{38mm}v(y)\nabla\bar u(x)+\bar u(x)\nabla v(y)\Big)\,\dd x\dd y\nonumber\\
			 &+4a\iint\left(\frac{x}{|x|^4}-\frac{y}{|y|^4}\right)
			      \cdot\nabla\rho(x-y)|u(x)|^2|v(y)|^2\,\dd x\dd y.
			\end{align}
			The bi-Laplacian contribution has therefore been included, rather
			than discarded.

			For a convex $\rho$, the Hessian form in
			\eqref{eq-interaction-square-full} is nonnegative. Integrating
			in time, we obtain
			\begin{align*}
			 &4\int_0^T\iint H_\rho(x-y)
			 \Big(\bar v(y)\nabla u(x)+u(x)\nabla\bar v(y),\\
			 &\hspace{40mm}v(y)\nabla\bar u(x)+\bar u(x)\nabla v(y)\Big)
			                                         \,\dd x\dd y\dd t\\
			 &\quad\leq |I_\rho'(0)|+|I_\rho'(T)|\\
			 &\qquad+4a\int_0^T\iint
			     \left|\left(\frac{x}{|x|^4}-\frac{y}{|y|^4}\right)
			                    \cdot\nabla\rho(x-y)\right|
			                         |u(x)|^2|v(y)|^2\,\dd x\dd y\dd t.
			\end{align*}

			For $a>0$, use Theorem~\ref{thm-locals} with $b=0$ and
			$\alpha=1/2$. The admissibility condition becomes
			\[
			 \frac12<\frac14+\frac12\mu_0(3,a)
			          =\frac14+\frac12\sqrt{\frac14+a},
			\]
			which holds for every $a>0$. Since no spherical harmonics are
			excluded when $b=0$, we have for arbitrary initial data
			\begin{equation}\label{eq-smoothing-positive-force}
			 \int_0^T\int_{\R^3}\frac{|u(t,x)|^2}{|x|^3}\,\dd x\dd t
			 \lesssim_a\|\La^{1/4}u(0)\|_2^2
			 \leq\! C_a\|u(0)\|_2\|u(0)\|_{H_a^1}.
			\end{equation}
			The last inequality follows from spectral Cauchy--Schwarz:
			$\|\La^{1/4}h\|_2^2\leq\|h\|_2\|\La^{1/2}h\|_2$.
			The same estimate holds for $v$. We can now estimate the force
			term, keeping the two contributions separate:
			\begin{align*}
			 &4a\int_0^T\iint
			     \left|\left(\frac{x}{|x|^4}-\frac{y}{|y|^4}\right)
			                     \cdot\nabla\rho(x-y)\right|
			                       |u(x)|^2|v(y)|^2\,\dd x\dd y\dd t\\
			 &\quad\leq4a\|\nabla\rho\|_\infty
			     \int_0^T\iint\left(|x|^{-3}+|y|^{-3}\right)
			                       |u(x)|^2|v(y)|^2\,\dd x\dd y\dd t\\
			 &\quad\leq4a\|\nabla\rho\|_\infty\Bigg(
			       \|v(0)\|_2^2\int_0^T\int_{\R^3}|x|^{-3}|u(x)|^2\,\dd x\dd t\\
			 &\hspace{54mm}+\|u(0)\|_2^2
			             \int_0^T\int_{\R^3}|y|^{-3}|v(y)|^2\,\dd y\dd t\Bigg)\\
			 &\quad\lesssim_a\|\nabla\rho\|_\infty
			 \Big(\|v(0)\|_2^2\|u(0)\|_2\|u(0)\|_{H_a^1}
			      +\|u(0)\|_2^2\|v(0)\|_2\|v(0)\|_{H_a^1}\Big).
			\end{align*}
			For $a=0$ the force term is identically zero, so no weighted
			smoothing estimate at $a=0$ is needed. In particular, we make no
			claim that \eqref{eq-smoothing-positive-force} holds at $a=0$.
			Combining the force estimate with
			\eqref{eq-one derivative forI}, for every $a\geq0$ we have
			\begin{align}\label{eq-Hessian}
			 &\int_0^T\iint H_\rho(x-y)
			 \Big(\bar v(y)\nabla u(x)+u(x)\nabla\bar v(y),\nonumber\\
			 &\hspace{38mm}v(y)\nabla\bar u(x)+\bar u(x)\nabla v(y)\Big)
			                                         \,\dd x\dd y\dd t\nonumber\\
			 &\quad\leq C\|\nabla\rho\|_\infty
			 \Big(\|v(0)\|_2^2\|u(0)\|_2\|u(0)\|_{H_a^1}
			      +\|u(0)\|_2^2\|v(0)\|_2\|v(0)\|_{H_a^1}\Big).
			\end{align}
			The constant is independent of $T$.

			Next, following \cite{P14}, define the convex function
			$\rho_M:\R\to\R$ by
			\[
			 \rho_M(z)=
			 \begin{cases}
			    \dfrac{Mz^2}{2}+\dfrac1{2M},& |z|\leq M^{-1},\\[1mm]
			    |z|,& |z|>M^{-1}.
			 \end{cases}
			\]
			Its first and second derivatives are
			\[
			 \rho_M'(z)=
			 \begin{cases}Mz,&|z|<M^{-1},\\
			               \operatorname{sgn}(z),&|z|>M^{-1},\end{cases}
			 \qquad
			 \rho_M''(z)=M\mathbf1_{\{|z|<M^{-1}\}}.
			\]
			The first derivative is continuous at $\pm M^{-1}$, so there are
			no boundary delta terms in the second derivative. In particular,
			$\|\rho_M'\|_\infty\leq1$. We first apply
			\eqref{eq-Hessian} to a convex smooth approximation of
			$\rho_M(x_1-y_1)$ and then pass to the limit. The Hessian has
			only one nonzero entry, namely
			$M\mathbf1_{\{|x_1-y_1|<M^{-1}\}}$ in the $(1,1)$ position.
			Consequently,
			\begin{align}\label{eq-direction-one-retained}
			 &M\int_0^T\iint_{|x_1-y_1|<M^{-1}}
			 \big|\bar v(y)\partial_{x_1}u(x)
			                    +u(x)\partial_{y_1}\bar v(y)\big|^2
			                                     \,\dd x\dd y\dd t\nonumber\\
			 &\quad\leq C\Big(
			    \|v(0)\|_2^2\|u(0)\|_2\|u(0)\|_{H_a^1}
			   +\|u(0)\|_2^2\|v(0)\|_2\|v(0)\|_{H_a^1}\Big).
			\end{align}
			There is no contribution from $|x_1-y_1|>M^{-1}$ because the
			Hessian vanishes there. Applying the same calculation to
			$\rho_M(x_j-y_j)$ for $j=2,3$ gives, for each $j=1,2,3$,
			\begin{align*}
			 &M\int_0^T\iint_{|x_j-y_j|<M^{-1}}
			 \big|\bar v(y)\partial_{x_j}u(x)
			                    +u(x)\partial_{y_j}\bar v(y)\big|^2
			                                     \,\dd x\dd y\dd t\\
			 &\quad\leq C\Big(
			    \|v(0)\|_2^2\|u(0)\|_2\|u(0)\|_{H_a^1}
			   +\|u(0)\|_2^2\|v(0)\|_2\|v(0)\|_{H_a^1}\Big).
			\end{align*}

			To pass from the three slabs to the ball, observe explicitly that
			\[
			 \{|x-y|<M^{-1}\}
			 \subset\bigcap_{j=1}^3\{|x_j-y_j|<M^{-1}\}.
			\]
			Since each integrand is nonnegative, we can restrict its domain
			to the ball and sum the three inequalities. This yields
			\begin{align*}
			 &M\int_0^T\iint_{|x-y|<M^{-1}}
			       |\bar v(y)\nabla_xu(x)+u(x)\nabla_y\bar v(y)|^2
			                                      \,\dd x\dd y\dd t\\
			 &\quad\leq M\sum_{j=1}^3\int_0^T
			      \iint_{|x_j-y_j|<M^{-1}}
			       |\bar v(y)\partial_{x_j}u(x)
			                 +u(x)\partial_{y_j}\bar v(y)|^2
			                                      \,\dd x\dd y\dd t\\
			 &\quad\lesssim
			    \|v(0)\|_2^2\|u(0)\|_2\|u(0)\|_{H_a^1}
			   +\|u(0)\|_2^2\|v(0)\|_2\|v(0)\|_{H_a^1}.
			\end{align*}
			Equivalently, on this ball the weight
			$\tilde\rho(z)=\sum_{j=1}^3\rho_M(z_j)$ has Hessian $M$ times
			the identity matrix and $\|\nabla\tilde\rho\|_\infty\leq\sqrt3$.
			This also proves \eqref{eq-claim H}.
			Now replace $u,v$ by $u_N,v_M$. Since $1\leq M\leq N$,
			\[
			 \|u_N(0)\|_{H_a^1}\lesssim N\|P_N^a\varphi_1\|_2,
			 \qquad
			 \|v_M(0)\|_{H_a^1}\lesssim M\|P_M^a\varphi_2\|_2.
			\]
			Dividing the preceding estimate by $M$, we obtain
			\begin{align}\label{eq-1morewarz reduce estimate}
			 &\int_0^T\iint_{|x-y|<M^{-1}}
			 \big|u_N(x)\nabla_y\bar v_M(y)
			              +\nabla_xu_N(x)\bar v_M(y)\big|^2
			                                      \,\dd x\dd y\dd t\nonumber\\
			 &\quad\leq C M^{-1}N
			       \|P_N^a\varphi_1\|_2^2\|P_M^a\varphi_2\|_2^2.
			\end{align}

			Next, $\La v_M$ is a solution with the same spectral support and
			$\|\La v_M(0)\|_2\lesssim M^2\|P_M^a\varphi_2\|_2$.
			Applying \eqref{eq-1morewarz reduce estimate} to this solution gives
			\begin{align}\label{eq-1morewarz-La-retained}
			 &\int_0^T\iint_{|x-y|<M^{-1}}
			 \big|u_N(x)\nabla_y(\La\bar v_M)(y)
			        +\nabla_xu_N(x)(\La\bar v_M)(y)\big|^2
			                                     \,\dd x\dd y\dd t\nonumber\\
			 &\quad\leq C M^{-1}N
			           \|P_N^a\varphi_1\|_2^2\|\La v_M(0)\|_2^2\nonumber\\
			 &\quad\leq C NM^3
			           \|P_N^a\varphi_1\|_2^2\|P_M^a\varphi_2\|_2^2.
			\end{align}
			Here $\La\bar v_M=\overline{\La v_M}$ because $\La$ has real
			coefficients; the solution to which the identity is applied is
			$\La v_M$, not its complex conjugate.

			We now combine the pointwise estimate with these two ball
			estimates. Pointwise in $(t,x,y)$,
			\begin{align*}
			 |\bar v_M(y)\nabla_xu_N(x)|^2
			 &\leq2|\bar v_M(y)\nabla_xu_N(x)
			                       +u_N(x)\nabla_y\bar v_M(y)|^2\\
			 &\quad+2|u_N(x)\nabla_y\bar v_M(y)|^2,\\
			 |(\La\bar v_M)(y)\nabla_xu_N(x)|^2
			 &\leq2|(\La\bar v_M)(y)\nabla_xu_N(x)
			                       +u_N(x)\nabla_y(\La\bar v_M)(y)|^2\\
			 &\quad+2|u_N(x)\nabla_y(\La\bar v_M)(y)|^2.
			\end{align*}
			Substitution into \eqref{eq-derivative for N}, followed by
			\eqref{eq-1morewarz reduce estimate} and
			\eqref{eq-1morewarz-La-retained}, yields
			\begin{align*}
			 &\int_0^T\int_{\R^3}|\bar v_M(x)\nabla_xu_N(x)|^2\,\dd x\dd t\\
			 &\quad\leq C\int_0^T\iint_{|x-y|<M^{-1}}
			 \Big(M^3|u_N(x)\nabla_y\bar v_M(y)|^2
			       +M^{-1}|u_N(x)\nabla_y(\La\bar v_M)(y)|^2\Big)
			                                                  \,\dd x\dd y\dd t\\
			 &\qquad+C\big(M^3(M^{-1}N)+M^{-1}(NM^3)\big)
			       \|P_N^a\varphi_1\|_2^2\|P_M^a\varphi_2\|_2^2\\
			 &\quad\leq C\int_0^T\iint_{|x-y|<M^{-1}}
			 \Big(M^3|u_N(x)\nabla_y\bar v_M(y)|^2
			       +M^{-1}|u_N(x)\nabla_y(\La\bar v_M)(y)|^2\Big)
			                                                  \,\dd x\dd y\dd t\\
			 &\qquad+CNM^2\|P_N^a\varphi_1\|_2^2\|P_M^a\varphi_2\|_2^2.
			\end{align*}

			It remains to bound both displayed error integrals. For the first,
			make the change of variables $w=y-x$. Translation invariance of
			the spatial $L^2$ norm, not of the Schr\"odinger equation, gives
			\begin{align}\label{eq-v_M estimate}
			 &M^3\int_0^T\iint_{|x-y|<M^{-1}}
			                    |u_N(x)\nabla_y\bar v_M(y)|^2\,\dd x\dd y\dd t\nonumber\\
			 &=M^3\int_{|w|<M^{-1}}\int_0^T\int_{\R^3}
			                  |u_N(x)|^2|\nabla\bar v_M(x+w)|^2\,\dd x\dd t\dd w\nonumber\\
			 &\leq M^3\int_{|w|<M^{-1}}
			        \|u_N\|_{L_t^2((0,T);L_x^\infty)}^2
			        \|\nabla v_M\|_{L_t^\infty((0,T);L_x^2)}^2\,\dd w\nonumber\\
			 &\lesssim M^3\int_{|w|<M^{-1}}
			         N\|u_N\|_{L_t^2((0,T);L_x^6)}^2
			         M^2\|v_M(0)\|_2^2\,\dd w\nonumber\\
			 &\lesssim M^3\big|B(0,M^{-1})\big|NM^2
			          \|P_N^a\varphi_1\|_2^2\|P_M^a\varphi_2\|_2^2\nonumber\\
			 &\lesssim NM^2\|P_N^a\varphi_1\|_2^2\|P_M^a\varphi_2\|_2^2.
			\end{align}
			Here we used the endpoint Strichartz estimate, Bernstein's
			inequality, the quadratic form bound, and
			$|B(0,M^{-1})|=\frac{4\pi}{3}M^{-3}$.

			For the second integral, retain the same argument with $v_M$
			replaced by $\La v_M$. Keeping track of its prefactor gives
			\begin{align}\label{eq-Av_M estimate}
			 &M^{-1}\int_0^T\iint_{|x-y|<M^{-1}}
			             |u_N(x)\nabla_y(\La\bar v_M)(y)|^2\,\dd x\dd y\dd t\nonumber\\
			 &=M^{-4}\left[M^3\int_0^T\iint_{|x-y|<M^{-1}}
			             |u_N(x)\nabla_y(\La\bar v_M)(y)|^2\,\dd x\dd y\dd t\right]\nonumber\\
			 &\lesssim M^{-4}NM^2
			                \|P_N^a\varphi_1\|_2^2\|\La v_M(0)\|_2^2\nonumber\\
			 &\lesssim M^{-4}NM^2M^4
			                \|P_N^a\varphi_1\|_2^2\|P_M^a\varphi_2\|_2^2\nonumber\\
			 &\lesssim NM^2\|P_N^a\varphi_1\|_2^2\|P_M^a\varphi_2\|_2^2.
			\end{align}
			Combining \eqref{eq-v_M estimate} and \eqref{eq-Av_M estimate}
			with the preceding bound proves \eqref{eq-H^1 estimate1}.
			Both gradient estimates are uniform in $T$, and the same argument
			applies on $(-T,0)$. Letting $T\to\infty$ gives the corresponding
			global estimates. Finally,
			\begin{align*}
			 \|\La^{1/2}(u_Nv_M)\|_2^2
			 &=\|\nabla(u_Nv_M)\|_2^2
			                   +a\||x|^{-1}u_Nv_M\|_2^2\\
			 &\leq(1+4a)\|\nabla(u_Nv_M)\|_2^2\\
			 &\leq2(1+4a)
			       \big(\|v_M\nabla u_N\|_2^2+\|u_N\nabla v_M\|_2^2\big).
			\end{align*}
			The absolute values of the factors in these two terms are
			unchanged when $v_M$ and $\nabla v_M$ are conjugated. Thus
			\eqref{eq-H^1 estimate1}--\eqref{eq-H^1 estimate2} imply
			\eqref{eq-A^1 estimate}, which completes the proof of the lemma.
		\end{proof}

		We now prove the $L^2$ bilinear estimate for arbitrary initial data.
		The divergence identity below eliminates the singular potential term,
		so no decomposition into spherical harmonics is required.

		\begin{theorem}
			\label{Theorem-linear bilinear estimate}
			Let $1 \leq M \leq N$ be dyadic integers and $a\geq0$. Then there exists $C$ such that
			\begin{align*}
				\big\| e^{it\mathcal{L}_a}P_N^a\varphi_1 e^{it\mathcal{L}_a}P_M^a\varphi_2 \big\|_{L_t^2(\R;L_x^2(\R^3))}^2 \leq CM^2N^{-1}\|P_N^a\varphi_1\|_{L_x^2(\R^3)}^2\|P_M^a\varphi_2\|_{L_x^2(\R^3)}^2
			\end{align*}
		\end{theorem}

		\begin{proof}
			
			For simplicity, we denote
 $u_N : = e^{it\mathcal{L}_a}P_N^a\varphi_1$, $v_M : = e^{it\mathcal{L}_a}P_M^a\varphi_2$.
			We perform a Littlewood--Paley decomposition of the product
			$v_Mu_N$, separating $K\leq N$ from $K>N$. Then
			\begin{align*}
			 \|v_Mu_N\|_{L_{t,x}^2}^2
			 &=\left\|\sum_{K\in2^\Z}P_K^a(v_Mu_N)\right\|_{L_{t,x}^2}^2\\
			 &\leq2\left\|\sum_{K\in2^\Z,\ K\leq N}
			                       P_K^a(v_Mu_N)\right\|_{L_{t,x}^2}^2
			     +2\left\|\sum_{K\in2^\Z,\ K>N}
			                       P_K^a(v_Mu_N)\right\|_{L_{t,x}^2}^2\\
			 &=:2I_1+2I_2.
			\end{align*}
			\textbf{Low frequency:} For $I_1$, let
			\[
			 S_N=\sum_{K\in2^\Z,\ K\leq N}P_K^a=P_{\leq N}^a.
			\]
			On the spectral support of $u_N$, define
			$\widetilde u_N=N^2\La^{-1}u_N$. Then
			$\widetilde u_N$ solves \eqref{linear eq}, has the same spectral
			support as $u_N$, and satisfies
			\[
			 \La\widetilde u_N=N^2u_N,
			 \qquad
			 \|\widetilde u_N(0)\|_2\lesssim\|P_N^a\varphi_1\|_2.
			\]
			Thus the original insertion of two derivatives becomes
			\begin{equation}\label{Sn}
			 S_N(v_Mu_N)=N^{-2}S_N(v_M\La\widetilde u_N).
			\end{equation}
			Instead of expanding $\La(v_M\widetilde u_N)$, subtract the
			term with $\La$ acting on $v_M$. Direct calculation gives
			\begin{align}\label{eq-divergence-cancellation-retained}
			 v_M\La\widetilde u_N-\widetilde u_N\La v_M
			 &=-v_M\Delta\widetilde u_N+\widetilde u_N\Delta v_M\nonumber\\
			 &=\operatorname{div}
			        (\widetilde u_N\nabla v_M-v_M\nabla\widetilde u_N).
			\end{align}
			The two multiplication-potential terms cancel identically,
			and so do the cross terms when the divergence is expanded.
			In particular, no estimate for $|x|^{-2}\widetilde u_Nv_M$
			is needed.

			We next prove the low-frequency divergence bound used to estimate
			\eqref{eq-divergence-cancellation-retained}. Since $S_N$ is a
			real spectral multiplier supported where $\sqrt{\La}\leq2N$,
			positivity of the quadratic form gives
			\[
			 \|\nabla S_Nh\|_2^2
			 \leq\|\La^{1/2}S_Nh\|_2^2
			 \lesssim N^2\|h\|_2^2.
			\]
			For $X\in L^2(\R^3;\C^3)$, integration by parts and duality yield
			\begin{align*}
			 |\langle S_N\operatorname{div}X,h\rangle|
			 &=|\langle X,-\nabla S_Nh\rangle|\\
			 &\leq\|X\|_2\|\nabla S_Nh\|_2
			 \lesssim N\|X\|_2\|h\|_2.
			\end{align*}
			Therefore, with the divergence understood distributionally,
			\begin{equation}\label{eq-low-divergence-retained}
			 \|S_N\operatorname{div}X\|_2\lesssim N\|X\|_2.
			\end{equation}

			Combining \eqref{Sn} and
			\eqref{eq-divergence-cancellation-retained}, we obtain
			\begin{align*}
			 S_N(v_Mu_N)
			 ={}&N^{-2}S_N\operatorname{div}
			          (\widetilde u_N\nabla v_M-v_M\nabla\widetilde u_N)\\
			 &+N^{-2}S_N(\widetilde u_N\La v_M).
			\end{align*}
			Consequently, by \eqref{eq-low-divergence-retained},
			\begin{align*}
			 I_1
			 &\lesssim N^{-2}
			   \left(\|\widetilde u_N\nabla v_M\|_{L_{t,x}^2}^2
			          +\|v_M\nabla\widetilde u_N\|_{L_{t,x}^2}^2\right)
			       +N^{-4}\|\widetilde u_N\La v_M\|_{L_{t,x}^2}^2.
			\end{align*}
			The two gradient terms are precisely the estimates
			\eqref{eq-H^1 estimate2} and \eqref{eq-H^1 estimate1}, applied to
			$\widetilde u_N,v_M$ (conjugation does not change their absolute
			values). Hence
			\begin{align*}
			 N^{-2}\left(\|\widetilde u_N\nabla v_M\|_{L_{t,x}^2}^2
			          +\|v_M\nabla\widetilde u_N\|_{L_{t,x}^2}^2\right)\lesssim N^{-2}M^2N
			          \|\widetilde u_N(0)\|_2^2\|v_M(0)\|_2^2
\lesssim M^2N^{-1}
			          \|P_N^a\varphi_1\|_2^2\|P_M^a\varphi_2\|_2^2.
			\end{align*}
			For the remaining term, Bernstein's inequality, the endpoint
			Strichartz estimate, and spectral localization give
			\begin{align*}
			 \|\widetilde u_N\La v_M\|_{L_{t,x}^2}^2
			\leq\|\widetilde u_N\|_{L_t^2L_x^\infty}^2
			            \|\La v_M\|_{L_t^\infty L_x^2}^2
			 \lesssim N\|\widetilde u_N\|_{L_t^2L_x^6}^2
			                   M^4\|v_M(0)\|_2^2
			 \lesssim NM^4\|P_N^a\varphi_1\|_2^2\|P_M^a\varphi_2\|_2^2.
			\end{align*}
			It follows, using $M\leq N$, that
			\begin{align*}
			 I_1
			 &\lesssim\big(M^2N^{-1}+M^4N^{-3}\big)
			               \|P_N^a\varphi_1\|_2^2\|P_M^a\varphi_2\|_2^2\\
			 &\lesssim M^2N^{-1}
			               \|P_N^a\varphi_1\|_2^2\|P_M^a\varphi_2\|_2^2.
			\end{align*}

			\textbf{High frequency:} We retain the high-output-frequency
			argument. The spectral supports of the dyadic multipliers have
			bounded overlap. Thus, by the spectral theorem and
			Lemma~\ref{le-improved bilinear estimates},
			\begin{align*}
			 I_2
			 &\lesssim\sum_{K\in2^\Z,\ K>N}
			               \|P_K^a(v_Mu_N)\|_{L_{t,x}^2}^2\\
			 &\lesssim\sum_{K\in2^\Z,\ K>N}K^{-2}
			               \|P_K^a\La^{1/2}(v_Mu_N)\|_{L_{t,x}^2}^2\\
			 &\lesssim N^{-2}\sum_{K\in2^\Z,\ K>N}
			               \|P_K^a\La^{1/2}(v_Mu_N)\|_{L_{t,x}^2}^2\\
			 &\lesssim N^{-2}\|\La^{1/2}(v_Mu_N)\|_{L_{t,x}^2}^2\\
			 &\lesssim M^2N^{-1}
			               \|P_N^a\varphi_1\|_2^2\|P_M^a\varphi_2\|_2^2.
			\end{align*}
			Combining the estimates for $I_1$ and $I_2$ completes the proof.
			All estimates can first be taken on finite time intervals and
			then passed to the global interval, since their constants are
			independent of the interval.
	\end{proof}
	We now remove the restriction $M\geq1$ by scaling and obtain
	the global bilinear estimate for all positive dyadic frequencies.
	\begin{theorem}
		\label{Theorem- bilinear estimate}
		Let $a\geq0$ and let $M,N\in2^\Z$ satisfy $M\leq N$. Then there exists $C$ such that
		\begin{align*}
			\big\| e^{it\mathcal{L}_a}P_N^a\varphi_1 e^{it\mathcal{L}_a}P_M^a\varphi_2 \big\|_{L_t^2(\R;L_x^2(\R^3))}^2 \leq CM^{2}N^{-1}\|P_N^a\varphi_1\|_{L_x^2(\R^3)}^2\|P_M^a\varphi_2\|_{L_x^2(\R^3)}^2.
		\end{align*}
	\end{theorem}
	\begin{proof}
		The estimate for $1\leq M\leq N$ follows from
		Theorem~\ref{Theorem-linear bilinear estimate}, without an angular
		restriction. To include every dyadic $0<M\leq N$, define the
		unitary dilation $D_\lambda h(x)=\lambda^{3/2}h(\lambda x)$.
		The homogeneity of the inverse-square potential gives
		\[
		 \La D_\lambda=\lambda^2D_\lambda\La,
		 \qquad
		 P_{\lambda K}^aD_\lambda=D_\lambda P_K^a,
		 \qquad
		 e^{it\La}D_\lambda=D_\lambda e^{i\lambda^2t\La}.
		\]
		Set $\lambda=M^{-1}$. Then
		$D_\lambda P_N^a\varphi_1$ and
		$D_\lambda P_M^a\varphi_2$ have frequencies $N/M$ and $1$,
		respectively, and their $L^2$ norms are unchanged. A change of
		variables gives
		\begin{align*}
		 &\big\|e^{it\La}D_\lambda P_N^a\varphi_1\,
		             e^{it\La}D_\lambda P_M^a\varphi_2\big\|_{L_{t,x}^2}^2\\
		 &\qquad=\lambda
		 \big\|e^{it\La}P_N^a\varphi_1\,
		             e^{it\La}P_M^a\varphi_2\big\|_{L_{t,x}^2}^2.
		\end{align*}
		Applying the already proved estimate at frequencies $1,N/M$
		and multiplying by $\lambda^{-1}=M$ yields exactly the desired
		factor $M^2N^{-1}$. This proves Theorem~\ref{thm1} as well.
		The same dilation extends \eqref{eq-A^1 estimate} to all positive
		dyadic frequencies; its squared left-hand side scales as
		$\lambda^3$ rather than $\lambda$.
	\end{proof}

	We also show the bilinear estimates involving the derivatives, which is crucial in establishing the global well-posedness in Theorem \ref{thm: nonlinearmain}.
	\begin{proposition}[Bilinear Strichartz estimates II]Let $\rho\in[0,\frac12)$, then
		\begin{gather}
			\big\|\La^\frac{\rho}{2}(e^{it\La}fe^{it\La}g)\big\|_{L_{t,x}^2([-T,T]\times\R^3)}\leq C\|f\|_{\dot H_a^{\frac12+\rho}(\R^3)}\|g\|_{L^2(\R^3)},\label{low}\\
			\big\|\La^\frac{\rho}{2}(e^{it\La}u_1e^{it\La}u_2)\big\|_{L_{t,x}^2([-T,T]\times\R^3)}\leq C\|u_1\|_{X_a^{\frac12+\rho,\frac12+}(I)}\|u_2\|_{X_a^{0,\frac12+}(I)}\label{low-B}
		\end{gather}
		Moreover, if $\frac12\leq\rho<\frac12+\sigma$ and $0\leq\sigma\leq1$, it holds
		\begin{gather}
			\big\|\La^\frac{\rho}{2}(e^{it\La}fe^{it\La}g)\big\|_{L_{t,x}^2([-T,T]\times\R^3)}\leq C\|f\|_{\dot H_a^{\sigma}(\R^3)}\|g\|_{\dot H_a^{\rho+\frac12-\sigma}(\R^3)},\label{high}\\
			\big\|\La^\frac{\rho}{2}(u_1u_2)\big\|_{L_{t,x}^2([-T,T]\times\R^3)}\leq C\|u_1\|_{X_a^{\sigma,\frac12+}([-T,T])}\|u_2\|_{X_a^{\rho+\frac12-\sigma,\frac12+}([-T,T])}\label{high-B}
		\end{gather}
		where $X_a^{s,b}([-T,T])$ is the restricted Bourgain space.
	\end{proposition}
	\begin{proof}
		First, we prove \eqref{low} and \eqref{high}. The remaining two estimates can be obtained by using the transfer principle as in \cite{BGT-Invent}. Throughout the proof, we use the notation $I=[-T,T]$ for short. Using the dyadic decomposition adapted to $\La$,
		\begin{align*}
			f=\sum_{N\in2^{\Bbb Z}}f_N:=\sum_{N\in2^{\Z}}P_N^af,\,\,		g=\sum_{N\in2^{\Bbb Z}}g_N:=\sum_{N\in2^{\Z}}P_N^ag,
		\end{align*}
		where $P_N^af:=\mathcal{F}_a^{-1}(\widehat{\psi}_N(\xi)\mathcal{F}_af(\cdot))(x)$ with $\supp\widehat{\psi}_N\subset \{\xi:|\xi|\sim N\}$. Hence, we have
		\begin{align*}
			&\hspace{5ex}	\big\|\La^\frac{\rho}{2}(e^{it\La}fe^{it\La}g)\big\|_{L_{t,x}^2(I\times\R^3)}
			\lesssim \sum_{N,M\in2^{\Bbb Z}}\big\|\La^\frac{\rho}{2}(e^{it\La}f_Ne^{it\La}g_M)\big\|_{L_{t,x}^2(I\times\R^3)}\\
			&\lesssim\sum_{N\in2^{\Bbb Z}}\sum_{M\leq N}N^\rho\big\|e^{it\La}f_Ne^{it\La}g_M\big\|_{L_{t,x}^2(I\times\R^3)}+\sum_{M\in2^{\Bbb Z}}\sum_{N\leq M}M^\rho\big\|e^{it\La}f_Ne^{it\La}g_M\big\|_{L_{t,x}^2(I\times\R^3)}.
		\end{align*}
		Note that
		\[
		M^\rho \cdot \frac{N}{M^\frac 12}\leq N^{\rho+\frac12}, \quad N \leq M, \quad 0 \leq \rho < \frac{1}{2},
		\]
		and
		\[
		N^{\rho} \cdot \frac{M}{N^\frac12} \leq N^{\rho+\frac{1}{2}}, \quad M \leq N, \quad 0 \leq \rho < \frac{1}{2}.
		\]
		Combining the above two bounds with Bernstein's inequality, one has
		\begin{align*}
			\big\|\La^\frac{\rho}{2}(e^{it\La}fe^{it\La}g)\big\|_{L_{t,x}^2(I\times\R^3)}&\lesssim\sum_{N,M\in2^{\Z}}N^{\rho+\frac12}\|f_N\|_{L^2(\R^3)}\|g_M\|_{L^2(\R^3)}\\
			&\leq C\|f\|_{\dot H_a^{\rho+\frac12}(\R^3)}\|g\|_{L^2(\R^3)},
		\end{align*}
		which completes the proof of \eqref{low}.
		Arguing like above and using the condition $\frac12\leq \rho<\frac12+\sigma$, we can prove the desired estimate \eqref{high}.
	\end{proof}
	\section{High-low frequency decomposition}
	In this section, we show the global well-posedness for \eqref{NLS} by adapting Bourgain's high-low decomposition argument to our setting.
	
	\noindent\textbf{Step 1.} (Decomposing the initial data).For initial data $u_0\in H_a^s(\R^d)$, we decompose it into
	\begin{align*}
		u_0=\phi+\psi:=P_{\leq N}^au_0+(I-P_{\leq N}^a)u_0.
	\end{align*}
	Using the distorted Fourier transform, it is not difficult to verify that
	\begin{align*}
		(\mathcal{F}_au_0)(\xi)=(\mathcal{F}_a\phi)(\xi)+(\mathcal{F}_a\psi)(\xi)=\chi_{|\xi|\leq N_0}\mathcal{F}_au_0+(1-\chi_{|\xi|\leq N_0})\mathcal{F}_au_0.
	\end{align*}
	By the Bernstein inequality and Plancherel theorem, we have
	\begin{align*}
		\big\|\phi(x)\|_{H_a^1(\R^d)}&=\|\phi\|_{L_x^2}+\|\xi\mathcal{F}_a\phi\|_{L_\xi^2}\lesssim N_0^{1-s},\quad 0<s<1.
	\end{align*}
	Using the Sobolev embedding $H_a^s\hookrightarrow L^4$ with $s\geq\frac{3}{4}$, we claim that
	\begin{align*}
		E_a(\phi)=\frac{1}{2}\int_{\R^3}|\nabla\phi(x)|^2+\frac{a}{|x|^2}|\phi(x)|^2+\frac{1}{4}\int_{\R^3}|\phi(x)|^4dx\lesssim N_0^{2(1-s)}.
	\end{align*}
	
	\noindent\textbf{Step 2.} (Estimate of low frequency evolution on $I=[0,\Delta T]$) In this step, we consider the cubic NLS with  initial data $\phi(x)$,
	\begin{align}\label{low-NLS}
		\begin{cases}
			i\partial_tu_1+\mathcal L_au_1=-|u_1|^2u_1,\\
			u_1(0,x)=\phi(x)\in H_a^1(\R^3).
		\end{cases}
	\end{align}
	By the classical Picard iteration, one can easily show that  \eqref{low-NLS} admits a local and strong solution $u_1(t)\in C(I,H_a^1(\R^3))\cap C^1(I,L^2(\R^3))$, where $I=[0,\Delta T]$ and $|I|\sim N_0^{-(1-s)\frac{9}{2}}$. Indeed,  by using the dispersive estimate and Sobolev embedding, one has
	\begin{align*}
		\|e^{-it\mathcal L_a}\phi\|_{L_{t,x}^5(I\times\R^3)}&\lesssim\big\|\La^\frac{3}{20}(e^{-it\La}\phi)\big\|_{L_t^5(I,L_x^\frac{10}{3}(\R^3))}\\
		&\lesssim|I|^\frac15\sup_{t\in I}\big\|\La^\frac{3}{20}(e^{-it\La}\phi)\big\|_{L_x^\frac{10}3(\R^3)}\lesssim |I|^\frac15\|\phi\|_{\dot H_a^\frac 9{10}(\R^3)}<|I|^\frac15 N_0^{\frac{9}{10}(1-s)},
	\end{align*}
	where we use the equivalence of Sobolev norms. Taking $|I|^\frac15=N_0^{\frac{9}{10}(1-s)-}$, we have that
	\begin{align}\label{linear-o(1)}
		\|\La^\frac{3}{20}(e^{-it\La}\phi(x))\|_{L_t^5(I,L_x^\frac{10}{3}(\R^3))}=o(1).
	\end{align}
	Therefore, by Sobolev embedding,
	\begin{align*}
		\|u_1(t)\|_{L_{t,x}^5(I\times\R^3)}\lesssim\|\La^\frac{3}{20}u_1(t)\|_{L_t^5(I,L_x^{\frac{10}{3}}(\R^3))},
	\end{align*}
	and by \eqref{linear-o(1)}, dispersive estimate and Hardy-Littlewood-Sobolev inequality, we get
	\begin{align*}
		\|\La^\frac{3}{20}u_1(t)\|_{L_t^5(I,L_x^{\frac{10}{3}}(\R^3))}&\lesssim o(1)+\Big\|\int_{0}^{t}|t-s|^{-\frac35}\||\nabla|^\frac{3}{10}u_1(t)\|_{L_x^\frac{10}{3}(\R^3)}\|u_1(t)\|_{L_x^5(\R^3)}^2ds\Big\|_{L_t^5}\\
		&\lesssim o(1)+\Big\|\int_{0}^{t}|t-s|^{-\frac35}\|\La^\frac{3}{20}u_1(t)\|_{L_x^\frac{10}{3}(\R^3)}\|u_1(t)\|_{L_x^5(\R^3)}^2ds\Big\|_{L_t^5}\\
		&\lesssim o(1)+\|\La^\frac{3}{20}u_1(t)\|_{L_t^5L_x^\frac{10}3(I\times\R^3)}\|u_1(t)\|_{L_{t,x}^5(I\times\R^3)}^2\\
		&\lesssim o(1)+\|\La^\frac{3}{20}u_1(t)\|_{L_t^5L_x^\frac{10}{3}(I\times\R^3)}^3.
	\end{align*}
	By continuity method, if we take $|I|$ sufficiently small, we  can bound this quantity by $o(1)$.
	
	Next, we show that the following $X^{s,b}$ bounds hold
	\begin{align}\label{claim}
		\|u_1(t)\|_{X_{a}^{0,\frac12+}(I)}<C,\quad \|u_1(t)\|_{X_a^{1,\frac12+}(I)}<N_0^{1-s}
	\end{align}
	for some positive constant $C>0$. From the Strichartz estimate,
	\begin{align*}
		\big\|e^{\pm it\mathcal L_a}f\big\|_{L_t^qL_x^r(I\times\R^3)}\lesssim\|f\|_{L^2(\R^3)},
	\end{align*}
	and Sobolev embedding in time $t$, we obtain
	\begin{align}\label{bourgain-1}
		\|F(t,x)\|_{L_t^qL_x^r(I\times\R^3)}&=\big\|e^{it\La}e^{-it\La}F(t,x)\big\|_{L_t^qL_x^r(I\times\R^3)}\lesssim \sup_{t\in I} \|e^{-it\mathcal{L}_a}F(t,x)\|_{L^2(\R^3)}\notag\\
		&\lesssim\|F(t,x)\|_{X_a^{0,\frac12+}(I)}.
	\end{align}
	By duality, we also have
	\begin{align*}
		\|F\|_{X_a^{0,-\frac12-}}\lesssim\|F(t,x)\|_{L_t^{q^\prime}L_x^{r^\prime}(I\times\R^3)},\quad\forall (q,r)\in\Lambda_0.
	\end{align*}
	Then interpolating with $\|F\|_{X_a^{0,0}}=\|F\|_{L_{t,x}^2}$, one has
	\begin{align*}
		\|F\|_{X_a^{0,-\frac12+}}\lesssim\|F(t,x)\|_{L_{t,x}^{\frac{10}{7}+}(I\times\R^3)},\quad\forall (q,r)\in\Lambda_0.
	\end{align*}
	Let $\L\in\{\La^\frac12,\langle\La\rangle^\frac12\}$, the analogue of \eqref{bourgain-1} holds
	\begin{align*}
		\big\|\L^\rho F(t,x)\big\|_{L_t^qL_x^r(I\times\R^3)}&=\big\|e^{it\La}\L^\rho e^{-it\La}F\big\|_{L_t^qL_x^r(I\times\R^3)}\lesssim  \sup_{t\in I}\big\|\L^\rho e^{-it\La}F\|_{L_x^2(\R^3)}\\
		&\lesssim\|F\|_{X_a^{\rho,\frac12+}}.
	\end{align*}
	Thanks to the Sobolev embedding, we actually have the following general estimates
	\begin{align*}
		\|F(t,x)\|_{L_t^{\tilde{q}}L_x^{\tilde{r}}(I\times\R^3)}\lesssim\|F\|_{X^{\rho,\frac12+}_a},\quad\frac{2}{\tilde q}=3\big(\frac12-\frac{\rho}{3}-\frac1{\tilde{r}}\big),\,\,\rho\in[0,1].
	\end{align*}
	Specifically, we have
	\begin{align}\label{Strichartz-Xsb}
		\|F\|_{L_{t,x}^5(I\times\R^3)}\lesssim\|F\|_{X_a^{\frac12,\frac12+}(I)},\,\,\|F\|_{L_{t,x}^{10}(I\times\R^3)}\lesssim\|F\|_{X_a^{1,\frac12+}(I)}
	\end{align}
	By the dual  Strichartz estimate in Bourgain space and the Gagliardo-Nirenberg inequality, one has
	\begin{align*}
		\big\|u_1(t)\big\|_{X_a^{1,\frac12+}(I)}&\lesssim\|\phi\|_{H_a^1(\R^3)}+\||u_1|^2\nabla u_1(t)\|_{X_a^{0,-\frac12+}(I)}\\
		&\lesssim\|\phi\|_{H_a^1(\R^3)}+\||u_1|^2\nabla u_1\|_{L_{t,x}^{\frac{10}{7}+}(I\times\R^3)}\\
		&\lesssim\|\phi\|_{H_a^1(\R^3)}+\|u_1\|_{L_{t,x}^5(I\times\R^3)}\|u_1\|_{L_{t,x}^{5+}(I\times\R^3)}\|\mathcal L_a^\frac12u_1\|_{L_{t,x}^{\frac{10}{3}}(I\times\R^3)}\\
		&\lesssim N_0^{1-s}+\|u_1(t)\|_{L_{t,x}^5(I\times\R^3)}^{2-\theta}\|u_1(t)\|_{L_{t,x}^{10}(I\times\R^3)}^\theta\|\La^\frac12 u_1(t)\|_{L_{t,x}^{\frac{10}{3}}(I\times\R^3)}\\
		&\lesssim N_0^{1-s}+o(1)\|u_1(t)\|_{X_a^{1,\frac12+}}^{1+\theta}
	\end{align*}
	for some $\theta\in(0,1)$. Similarly, by using the Strichartz estimate, H\"older's inequality and Bernstein, we claim that
	\begin{align*}
		\|u_1(t)\|_{X_a^{0,\frac12+}(I)}\lesssim \|\phi\|_{L^2}+o(1)\|u_1(t)\|_{X_a^{0,\frac12+}(I)}.
	\end{align*}
	Hence, we have proved two estimates in \eqref{claim}.
	
	\noindent\textbf{Step 3.} (A priori bound of the high-frequency evolution) Let $u(t)=u_1(t)+u_2(t),$ then $u_2(t)$ satisfies the difference evolution equation
	\begin{align*}
		\begin{cases}
			i\partial_tu_2+\mathcal{L}_au_2=-\big(2|u_1|^2u_2+|u_2|^2u_2+2u_1^2\overline{u_2}+2\overline{u_1}u_2^2+2u_1|u_2|^2\big),\\
			u_2(0,x)=\psi(x).
		\end{cases}
	\end{align*}
	Using the Duhamel formula, it can be represented as
	\begin{align*}
		u_2(t)=e^{it\mathcal{L}_a}\psi(x)+w(t).
	\end{align*}
	Notice that
	\begin{align*}
		\|\psi\|_{H_a^s(\R^3)}\lesssim C,\,\,\|\psi\|_{L^2(\R^3)}\lesssim N_0^{-s},
	\end{align*}
	combining with \eqref{Strichartz-Xsb} and H\"older's inequality, the following estimates hold for $s>\frac12$,
	\begin{align}\label{key-1}
		\|u_2(t)\|_{X_a^{0,\frac12+}(I)}&\lesssim\|\psi\|_{L^2}+\big\|\||u_2|^2u_2+2|u_1|^2u_2+2u_1^2\overline{u_2}+2\overline{u_1}u_2^2+2u_1|u_2|^2\big\|_{X_a^{0,-\frac12+}(I)}\\
		&\lesssim\|\psi\|_{L^2(\R^3)}+\|u_2\|_{X_a^{0,\frac12+}(I)}(\|u_2\|_{X_a^{\frac12,\frac12+}(I)}^2+o(1)+o(1)\|u_2\|_{X_a^{\frac12,\frac12+}(I)})
	\end{align}
	\begin{align}\label{key-2}
		\|u_2(t)\|_{X_a^{s,\frac12+}(I)}&\lesssim\|\psi\|_{H_a^s}+\big\|\||u_2|^2u_2+2|u_1|^2u_2+2u_1^2\overline{u_2}+2\overline{u_1}u_2^2+2u_1|u_2|^2\big\|_{X_a^{s,-\frac12+}(I)}\\
		&\lesssim\|\psi\|_{H_a^s(\R^3)}+\|u_2\|_{X_a^{s,\frac12+}(I)}(\|u_2\|_{X_a^{\frac12,\frac12+}(I)}^2+o(1)+o(1)\|u_2\|_{X_a^{\frac12,\frac12+}(I)})\\
		&\hspace{2ex}+\|u_1\|_{X_a^{s,\frac12+}(I)}(o(1)\|u_2\|_{X_a^{\frac12,\frac12+}(I)}+\|u_2\|_{X_a^{\frac12,\frac12+}(I)}).
	\end{align}
	By interpolation, we have
	\begin{align*}
		\|u_1\|_{X_a^{s,\frac12+}(I)}\lesssim\|u_1\|_{X_a^{0,\frac12+}(I)}^{1-a}\|u_1\|_{X_a^{1,\frac12+}(I)}^a\lesssim N_0^{s(1-s)},
	\end{align*}
	where $a\in(0,1)$. Then invoking the above bound in \eqref{key-1} and \eqref{key-2}, we have
	\begin{align*}
		\|u_2\|_{X_a^{0,\frac12+}}\leq C,\,\,\|u_2\|_{X_a^{s,\frac12+}}\leq C.
	\end{align*}
	
	Next, we begin to estimate the energy increment of the difference equation since it is not  Hamiltonian. By the direct calculation and the equivalent of Sobolev norms,
	\begin{align*}
		\|\La^\frac12w(t)\|_{L_t^\infty L_x^2(I\times\R^3)}
		&\leq \|\La^\frac12w(t)\|_{X_a^{0,\frac12+}(I)}\\
		&\leq \sup_{\|W\|_{X_a^{0,\frac12+}}\leq 1}\int_{\R}(W,\La^\frac12(2|u_1|^2u_2+|u_2|^2u_2+2u_1^2\overline{u_2}+2\overline{u_1}u_2^2+2u_1|u_2|^2))\,ds\\
	\end{align*}
	Then by Plancherel's theorem, we have
	\begin{align*}
		(f,|\nabla| g)&=\int_{\R^3}f\overline{|\nabla| g}\,dx=\int_{\R^3}\widehat{f}\cdot|\xi|\overline{\widehat{g}}\,d\xi\\
		&=\int_{\R^3}|\xi|^\rho\widehat{f}\cdot|\xi|^{1-\rho}\overline{\widehat{g}}\,d\xi=(|\nabla|^\rho f,|\nabla|^{1-\rho} g).
	\end{align*}
	For convenience, we only estimate the nonlinear term $|u_1|^2u_2$.
	The basic computation yields
	\begin{align*}
		\|\La^\frac{1}{2}(|u_1|^2u_2)\|_{X_a^{0,-\frac12-}}&\lesssim	\|\nabla(|u_1|^2u_2)\|_{X_a^{0,-\frac12-}}\\
		&=\sup_{\|G\|_{X^{0,\frac12+}(I)}\lesssim1}|(G,u_1u_2\nabla u_1+u_1^2\nabla u_2)|\\
		&\lesssim\int_{I\times\R^3}|G||u_1||u_2|\nabla u_1|\,dx\,dt+
		|(\nabla(u_1u_2),u_1G)|\\
		&\lesssim\|G\|_{X^{0,\frac12+}(I)}\|u_1(t)\|_{X_a^{1,\frac12+}(I)}^2\|u_2(t)\|_{X_a^{0,\frac12+}(I)}+\int_{I\times\R^3}||\nabla|^\rho(u_1G)||\nabla|^{1-\rho}(u_2u_1)|\,dx\,dt\\
		&\lesssim N_0^{2-3s}+\||\nabla|^\rho(u_1G)\|_{L_{t,x}^2}\||\nabla|^{1-\rho}(u_1u_2)\|_{L_{t,x}^2}.
	\end{align*}
	Using the equivalence of Sobolev norm and bilinear Strichartz estimate, we have
	\begin{align*}
		\||\nabla|^\rho(u_1G)\|_{L_{t,x}^2}\||\nabla|^{1-\rho}(u_1u_2)\|_{L_{t,x}^2} &\leq \| \La^{\frac{\rho}{2}} ( u_1G) \|_{L_{t,x}^2} \| \La^{\frac{1-\rho}{2}}(u_1u_2) \|_{L_{t,x}^2} \\
		&\leq \| u_1 \|_{X_a^{\frac{1}{2}+\rho,\frac{1}{2}+}} \| u_2 \|_{X_a^{\frac{1}{2}-\rho,\frac{1}{2}+}} \| u_1 \|_{X_a^{\frac{1}{2},\frac{1}{2}+}} \\
		&\leq N_0^{2-3s+},
	\end{align*}
	where we complete the estimate of $\La(|u_1|^2u_2)$. As a conclusion, if $s>\frac23$, we have
	\begin{align*}
		\sup_{t\in I}\|w(t)\|_{H_a^1(\R^3)}\lesssim N_0^{2-3s+}.
	\end{align*}
	
	\noindent\textbf{Step 4.} (Iteration Scheme) Consider $t_1=\Delta T$, we have
	\begin{align*}
		u(t_1)=u_1(t_1)+e^{it_1\La}\psi(x)+w(t_1)\stackrel{\triangle}{=}\phi_1(t_1)+\psi_1(t_1),
	\end{align*}
	where
	\begin{align*}
		\phi_1(t_1):=u_1(t_1)+w(t_1),\,\,	\psi_1(t_1):=e^{it_1\La}\psi(x).
	\end{align*}
	Thus, the pair $(\phi, \psi)$ can be replaced by $(\phi_1(t_1), \psi_1(t_1))$ to re-iterate above steps. The Hamiltonian increment when replacing $\phi$ by $\phi_1(t_1)$ can be bounded by
	
	\begin{align*}
		\left|E_a(\phi_1(t_1)) - E_a(\phi_0)\right| &= \left|E_a(u_1(t_1) + \omega(t_1)) - E_a(u_1(t_1))\right| \\
		&\leq \left(\|u_1(t_1)\|_{H_a^1} + \|\omega(t_1)\|_{H_a^1}\right) \|\omega(t_1)\|_{H_a^1} \\
		&\quad + \left(\|u_1(t_1)\|_{L^6(\R^3)} + \|\omega(t_1)\|_{L^6(\R^3)}\right)^3 \|\omega(t_1)\|_{L^2(\R^3)} \\
		&\leq N_0^{1-s} N_0^{2-3s+} + N_0^{3(1-s)-s} \\
		&\leq N_0^{3-4s+},
	\end{align*}
	The energy increment between $\psi_1(t_1)$ and $\psi$ can be estimated in a similar way. Then, in order to extend the solution to any time $T > 0$, we need iterate $k$
	times where
	\begin{align*}
		k\sim \frac{T}{\Delta T}\sim TN_0^{\frac{9}{2}(1-s)-}.
	\end{align*}
	This leads to the restriction condition on regularity of initial data
	\begin{align*}
		TN_0^{\frac{9}{2}(1-s)-}N_0^{3-4s+}
		<E_a(\phi) \sim N_0^{2(1-s)},
	\end{align*}
	which is equivalent to $TN_0^{\frac{11-13s}{2}+}<1.$ Hence, if
	$s >\frac{11}{13}$,
	we can continue the iteration on $[0, T]$ and maintain the increment
	$N_0 := N_0(T) = T^{\frac{ 2}
		{13s-11} +}.$

	\vspace{2ex}
	\noindent\textbf{Statements and Declarations:} The authors have no relevant financial or non-financial interests to disclose. Also the authors do not use any LLM. 
	
	\noindent\textbf{Data availability:} Data sharing is not applicable to this article as no datasets were generated or analyzed during the current study.
	
	\noindent\textbf{Conflicts of interests:} On behalf of all authors, the corresponding author states that there is no conflict of interest.

\end{document}